\documentclass[11pt]{article}
\usepackage{amssymb,amsmath}
\usepackage[mathscr]{eucal}
\usepackage{graphicx}
\usepackage[cm]{fullpage}
\usepackage[english]{babel}
\usepackage[latin1]{inputenc}
\usepackage{color}
\usepackage{amsthm} 

\def\dom{\mathop{\mathrm{Dom}}\nolimits}
\def\im{\mathop{\mathrm{Im}}\nolimits}
\def\ker{\mathop{\mathrm{Ker}}\nolimits}
\def\rank{\mathop{\mathrm{rank}}\nolimits}
\def\lb{\mathop{\mathrm{Lb}}\nolimits}
\def\ub{\mathop{\mathrm{Ub}}\nolimits}
\def\T{\mathscr{T}}
\def\PT{\mathscr{PT}}
\def\I{\mathscr{I}}
\def\S{\mathscr{S}}
\def\GRR{\mathcal{R}}
\def\GRL{\mathcal{L}}
\def\GRH{\mathcal{H}}
\def\GRD{\mathcal{D}}
\def\GRJ{\mathcal{J}}

\def\N{\mathbb N}
\def\R{\mathbb R}
\def\Q{\mathbb Q}
\def\Z{\mathbb Z}

\def\OO{\mathscr{O}}

\def\PO{\mathscr{PO}}
\def\OP{\mathscr{OP}}
\def\POP{\mathscr{POP}}
\def\POI{\mathscr{POI}}
\def\POPI{\mathscr{POPI}}
\def\cl#1{\overline{#1}}
\def\reg{\mathop{\mathrm{Reg}}\nolimits}

\newtheorem{theorem}{Theorem}[section]
\newtheorem{proposition}[theorem]{Proposition}
\newtheorem{corollary}[theorem]{Corollary}
\newtheorem{lemma}[theorem]{Lemma}

\newtheorem{example}[theorem]{Example}
\newtheorem{question}[theorem]{Question}
\newtheorem{problem}[theorem]{Problem} 
\newtheorem{definition}[theorem]{Definition}

\newenvironment{proof*}[1]{\begin{trivlist}\item[\hskip%
\labelsep{\bf #1.}]}%
{\qed\rm\end{trivlist}}

\title{On orientation-preserving transformations of a chain}
\author{V.H. Fernandes\footnote{This work was developed within the 
FCT Project UID/MAT/00297/2013 of CMA and of Departamento de Matem\'atica da Faculdade de Ci\^encias e Tecnologia da Universidade Nova de Lisboa.},\, M.M. Jesus\footnote{This work was developed within the 
FCT Project UID/MAT/00297/2013 of CMA and of Departamento de Matem\'atica da Faculdade de Ci\^encias e Tecnologia da Universidade Nova de Lisboa.}\, and B. Singha}

\newcommand{\lastpage}{\addresss}

\newcommand{\addresss}{\small \sf  
\noindent{\sc V\'\i tor H. Fernandes}, 
CMA, Departamento de Matem\'atica, 
Faculdade de Ci\^encias e Tecnologia, 
Universidade NOVA de Lisboa, 
Monte da Caparica, 
2829-516 Caparica, 
Portugal; 
e-mail: vhf@fct.unl.pt. 

\medskip

\noindent{\sc Manuel M. Jesus}, 
CMA, Departamento de Matem\'atica, 
Faculdade de Ci\^encias e Tecnologia, 
Universidade NOVA de Lisboa, 
Monte da Caparica, 
2829-516 Caparica, 
Portugal; 
e-mail: mrj@fct.unl.pt. 

\medskip

\noindent{\sc Boorapa Singha},  
Department of Mathematics and Statistics, 
Faculty of Science and Technology, 
Chiang Mai Rajabhat University, 
Chiang Mai 50300, 
Thailand;
email: boorapas@yahoo.com.  
}

\begin{document}

\maketitle \vspace*{-1.5cm}

\renewcommand{\thefootnote}{}

\footnote{2010 \emph{Mathematics Subject Classification}: 20M10, 20M20.}

\footnote{\emph{Key words}: transformation semigroups, orientation-preserving, order-preserving.}

\renewcommand{\thefootnote}{\arabic{footnote}}
\setcounter{footnote}{0}

\begin{abstract}
In this paper we introduce the notion of an orientation-preserving transformation on an arbitrary chain, as 
a natural extension for infinite chains of the well known concept for finite chains introduced in 1998 by McAlister \cite{McAlister:1998} and, independently, in 1999 by Catarino and Higgins \cite{Catarino&Higgins:1999}. 
We consider the monoid $\POP(X)$ of all orientation-preserving partial transformations on a finite or infinite chain $X$ and its submonoids 
$\OP(X)$ and $\POPI(X)$ of all orientation-preserving full transformations and of all orientation-preserving partial permutations on $X$, respectively. 
The monoid $\PO(X)$ of all order-preserving partial transformations on $X$ and its injective counterpart $\POI(X)$ are also considered. 
We study the regularity and give descriptions of the Green's relations of the monoids $\POP(X)$, $\PO(X)$, $\OP(X)$, $\POPI(X)$ and $\POI(X)$.  
\end{abstract}

\section{Introduction and preliminaries}\label{pre}

Let $X$ be an arbitrary chain (finite or infinite). 

Denote by $\PT(X)$ the monoid of all partial transformations on $X$
(under composition of maps) and by $\T(X)$ its submonoid of all full transformations on $X$. 
Let $\alpha\in\PT(X)$. 
We say that $\alpha$ is \textit{order-preserving} if $x\le y$
implies $x\alpha\le y\alpha$, for all $x,y\in\dom(\alpha)$. 
For $Y\subseteq\dom(\alpha)$, the transformation 
$\alpha$ is said to be \textit{order-preserving on $Y$} if its restriction to $Y$ is order-preserving.  
Denote by $\PO(X)$ the submonoid of $\PT(X)$ of all
order-preserving partial transformations, i.e.
$$
\PO(X)=\{\alpha\in\PT(X)\mid \mbox{$x\le y$ implies $x\alpha\le
y\alpha$, for all $x,y\in \dom(\alpha)$}\}, 
$$
and by $\OO(X)$ the submonoid of $\PO(X)$ of all order-preserving
full transformations, i.e. 
$$
\OO(X)=\{\alpha\in\T(X)\mid \mbox{$x\le y$ implies $x\alpha\le
y\alpha$, for all $x,y\in X$}\} . 
$$

For a finite chain $X$, it is well known, and easy to prove, that $\OO(X)$ is a regular semigroup \cite{Gomes&Howie:1992}. 
The problem for an infinite chain $X$ is much more involved. 
Nevertheless,  a characterization of those posets $P$ 
for which the semigroup of all endomorphisms of $P$ is regular was done by A\v\i zen\v stat in 1968 
\cite{Aizenstat:1968} (see also the paper \cite{Adams&Gould:1989} by Adams and Gould).  
A description of the regular elements of $\OO(X)$ was given in 2010 by Mora and Kemprasit \cite{Mora&Kemprasit:2010} and 
the largest regular subsemigroup of $\OO(X)$ was characterized by Fernandes et al. in 
\cite{Fernandes&al:2014}. In this last paper, the authors also described the Green's relations on $\OO(X)$. 
For a chain $X_n$ with $n$ elements, e.g. $X_n=\{1<2<\cdots<n\}$, the monoid $\OO(X_n)$, 
usually denoted by $\OO_n$, has been extensively studied since the sixties. See 
\cite{Aizenstat:1962,Aizenstat:1962b,Fernandes:1997,Fernandes:2002,Fernandes&al:2010,Fernandes&Volkov:2010,Gomes&Howie:1992,
Higgins:1995,Howie:1971,Repnitskii&Vernitskii:2000,Repnitskii&Volkov:1998,Vernitskii&Volkov:1995}.

\smallskip 

Let $a=(a_1,a_2,\ldots,a_t)$ be a sequence of $t$ ($t\geq0$)
elements from the chain $X_n$. We say that $a$ is \textit{cyclic} 
if there exists no more than one index
$i\in\{1,\ldots,t\}$ such that $a_i>a_{i+1}$,
where $a_{t+1}$ denotes $a_1$.
An element $\alpha\in\T(X_n)$ is called an \textit{orientation-preserving} 
transformation if the sequence of its
images $(1\alpha,\ldots,n\alpha)$ is cyclic.
It is a routine to check that the product
of two orientation-preserving transformations on $X_n$ is an
orientation-preserving transformation \cite{Catarino&Higgins:1999}.
Denote by $\OP_n$ the submonoid of $\T(X_n)$  whose
elements are orientation-preserving.
The notion of an orientation-preserving
transformation on a finite chain was introduced by McAlister in \cite{McAlister:1998} and,
independently, by Catarino and Higgins in \cite{Catarino&Higgins:1999}. 
Several properties of the monoid $\OP_n$ have been investigated in 
these two papers. A presentation for the monoid $\OP_n$, in terms of $2n-1$ generators, was given by Catarino in \cite{Catarino:1998}. Another presentation for $\OP_n$, in terms of $2$ (its rank) generators, was found by
Arthur and Ru\v{s}kuc \cite{Arthur&Ruskuc:2000}.
The congruences of $\OP_n$ were completely described by Fernandes et al. in \cite{Fernandes&Gomes&Jesus:2009}. 
Semigroups of orientation-preserving transformations were also studied in several other recent papers 
(e.g. see \cite{Araujo&al:2011,Dimitrova&al:2012,
Fernandes&Gomes&Jesus:2011,Fernandes&al:2016,Fernandes&Quinteiro:2011,Fernandes&Quinteiro:2012,
Fernandes&Quinteiro:2014,Zhao&Fernandes:2015}). 

\smallskip 

This paper is organized as follows. 
In the remaining part of the Section \ref{pre} we define orientation-preserving transformations on an arbitrary chain and present some basic properties, 
in particular we show that the set of all orientation-preserving transformations forms a semigroup. 
In Section \ref{reg} we give a criterion for the regularity of an orientation-preserving full transformation, to whose proof we dedicate all Section \ref{proof}. 
Finally, in Section \ref{RGR} we give descriptions of the Green's relations for the various semigroups considered. 

\medskip 

\begin{definition}\label{opdef}
Let $\alpha\in\PT(X)$.  We say that $\alpha$  is \textit{orientation-preserving} if $\alpha$ is the empty transformation or if there exists
a non-empty subset $Y$ of $\dom(\alpha)$ such that:
\begin{description}
\item (OP1) $\alpha$ is order-preserving both on $Y$ and on $\dom(\alpha)\setminus Y$; 
\item (OP2) For all $a\in Y$ and $b\in\dom(\alpha)\setminus Y$, we have $a\le b$ and $a\alpha\ge b\alpha$.
\end{description}
If $\alpha\ne\emptyset$, we call such a subset $Y$ an \textit{ideal} of $\alpha$. For $\alpha=\emptyset$, we define the \textit{ideal} of $\alpha$ as being the empty set. 
\end{definition}

Notice that, in the previous conditions, $Y$ is an order ideal of $\dom(\alpha)$ and $\dom(\alpha)\setminus Y$ is an order filter of $\dom(\alpha)$.

An orientation-preserving transformation can be represented pictorially as a mapping with the following geometrical property: representing in the plane its domain and its image as two concentric circles, it is possible to connect by means of continuous lines, which do not pairwise intersect, each element of its domain with the respective image (see Figure \ref{fig}). 
\begin{figure}[h] 
\centering \label{fig}
\includegraphics[scale=.15]{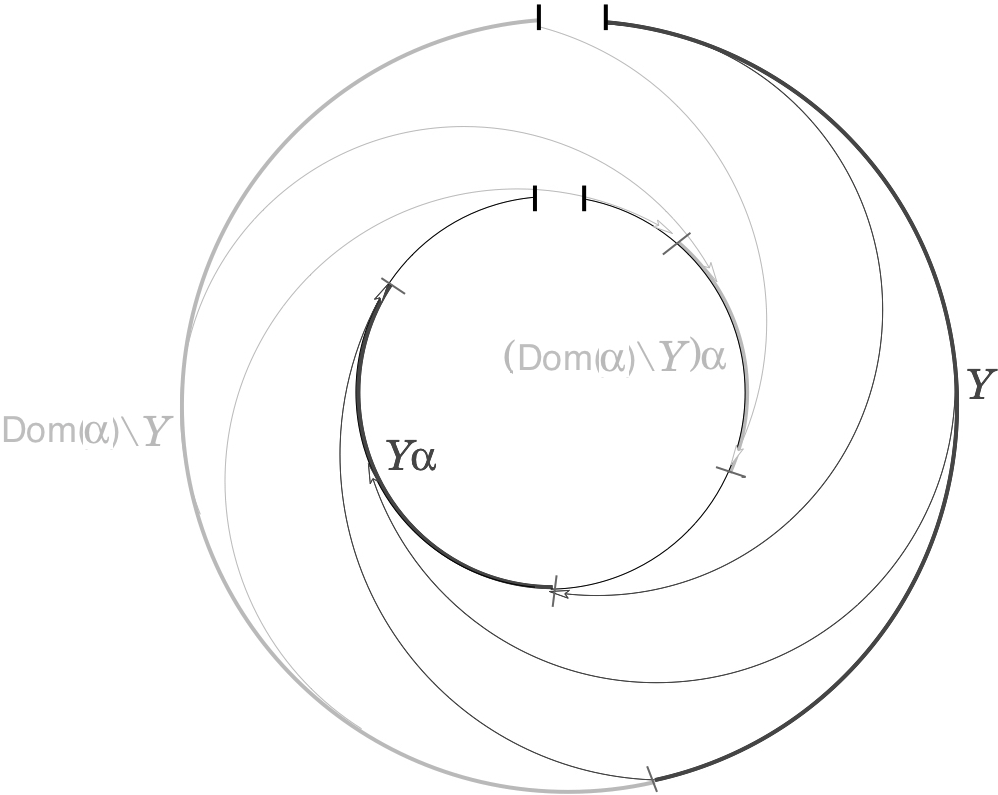}
\caption{An orientation-preserving transformation $\alpha$}
\end{figure}

For a finite chain, it is easy to realize that the above notion coincides with the notion of an orientation-preserving transformation introduced by McAlister \cite{McAlister:1998} and by Catarino and Higgins \cite{Catarino&Higgins:1999}. 
Furthermore, it is a natural extension for infinite chains. 

Denote by $\POP(X)$ the subset of $\PT(X)$ of all orientation-preserving partial transformations. 
Also, denote by $\OP(X)$ the subset of $\T(X)$ of all orientation-preserving full transformations, i.e. 
$\OP(X)=\POP(X)\cap\T(X)$. Clearly, for
$\alpha\in\PT(X)$, we have $\alpha\in\PO(X)$ if and only if $\alpha\in\POP(X)$
and $\alpha$ admits $\dom(\alpha)$ as an ideal. In particular, we have $\PO(X)\subseteq\POP(X)$ and $\OO(X)\subseteq\OP(X)$.

Also clear is the following lemma:

\begin{lemma}\label{non-c}
Let $\alpha\in \POP(X)$ admitting 
a non-empty proper subset $Y$ of $\dom(\alpha)$ as an ideal.  
Then $\alpha$ is non-constant if and only if there
exist $a\in Y$ and $b\in\dom(\alpha)\setminus Y$ such that $a\alpha > b\alpha$.
\end{lemma}

Observe that any constant mapping of $\PT(X)$ is order-preserving and so it is also 
orientation-preserving. Moreover, it admits
as an ideal any non-empty order ideal of its domain.  
Next, we show that for non-constant mappings the situation is completely different.

\begin{proposition}
Let $\alpha$ be a non-constant element of $\POP(X)$. 
Then $\alpha$ admits a unique ideal.
\end{proposition}
\begin{proof}
If $\alpha$ is the empty transformation then, by definition, it has only the empty set as an ideal. 
If $\alpha\in\PO(X)$ then $\dom(\alpha)$
is the unique ideal of $\alpha$. In fact, if $\alpha$ would admit a
non-empty proper subset $Y$ of $\dom(\alpha)$ as an ideal
then, by Lemma \ref{non-c}, there exist $a\in Y$ and $b\in \dom(\alpha)\setminus Y$ 
such that $a\alpha > b\alpha$. Since $a\le b$, by (OP2), then $\alpha$ would not be
order-preserving.

Next, suppose that $\alpha\not\in\PO(X)$ and
let $Y$ and $Z$ be two ideals of $\alpha$. 
Then $Y$ and $Z$ are non-empty proper subsets of $\dom(\alpha)$. 
By Lemma \ref{non-c}, we may consider elements $a\in Z$ and $b\in \dom(\alpha)\setminus Z$ 
such that $a\alpha > b\alpha$. Furthermore, $a < b$, by (OP2). 

Take $x\in Y$ and, by contradiction, suppose that $x\not\in Z$. Then, by (OP2), $a < x$ and
$a\alpha\ge x\alpha$ (since $a\in Z$ and $x\in \dom(\alpha)\setminus Z$). 
From $a < x$ and $x\in Y$, it
follows that $a\in Y$, whence $a\alpha\le x\alpha$. Thus
$a\alpha=x\alpha$. As $b,x\in
\dom(\alpha)\setminus Z$, if $x\le b$ then $a\alpha=x\alpha\le
b\alpha < a\alpha$, a contradiction. Then $b < x$ and, since $x\in
Y$, it follows that $b\in Y$. Thus $a,b\in Y$ and $a < b$, whence
$a\alpha\le b\alpha$, again a contradiction. Therefore $x\in Z$
and so $Y\subseteq Z$.

Similarly, we prove that $Z\subseteq Y$, whence $Y=Z$, as
required.
\end{proof}

Another property easy to prove of an orientation-preserving transformation is the following:

\begin{proposition}\label{glued}
Let $\alpha$ be an element of $\POP(X)$ with
ideal $Y$. If $Y\alpha\cap(\dom(\alpha)\setminus Y)\alpha\neq\emptyset$ then
$Y\alpha\cap(\dom(\alpha)\setminus Y)\alpha=\{m\}$, for some $m\in X$. 
Moreover, in this case, $Y\alpha$
has a minimum element, $(\dom(\alpha)\setminus Y)\alpha$ has a maximum element and both
of these elements coincide with $m$.
\end{proposition}
\begin{proof}
Let  $m\in Y\alpha\cap(\dom(\alpha)\setminus Y)\alpha$ and take $a\in Y$ and
$b\in \dom(\alpha)\setminus Y$ such that $a\alpha=m=b\alpha$. 
Let $c\in Y\alpha$. Then $c=y\alpha$, for
some $y\in Y$. Since $b\in\dom(\alpha)\setminus Y$, we have $c=y\alpha\ge b\alpha=m$ and so
$m$ is the minimum of $Y\alpha$. 
Moreover, this also proves that 
$Y\alpha\cap(\dom(\alpha)\setminus Y)\alpha=\{m\}$, by the uniqueness of a minimum. 
On the other hand, let $d\in(\dom(\alpha)\setminus Y)\alpha$. 
Then $d=z\alpha$, for some $z\in\dom(\alpha)\setminus Y$. 
Now, since $a\in Y$, we have
$m=a\alpha\ge z\alpha=d$ and so $m$ is also the maximum of
$(\dom(\alpha)\setminus Y)\alpha$, as required.
\end{proof}

Now, as an application of the previous proposition, 
we present a property of the idempotents of $\OP(X)$ that we will use in Section \ref{proof}. 

\begin{proposition}\label{idps}
Let $\alpha$ be an idempotent of $\OP(X)$ and let $Y$ be an ideal of $\alpha$. 
Then, we have one and only one of the following conditions: 
\begin{enumerate}
\item $\alpha\in\OO(X)$;
\item $\im(\alpha)=Y\alpha\subseteq Y$, $\min\im(\alpha)$ exists and $(X\setminus Y)\alpha=\{\min\im(\alpha)\}$; 
\item $\im(\alpha)=(X\setminus Y)\alpha\subseteq X\setminus Y$, $\max\im(\alpha)$ exists and $Y\alpha=\{\max\im(\alpha)\}$. 
\end{enumerate}
\end{proposition}
\begin{proof}
First, notice that, since $\alpha$ is idempotent, then $\alpha$ is order-preserving on $\im(\alpha)$. 
Hence, we have $\im(\alpha)\subseteq Y$ or $\im(\alpha)\subseteq X\setminus Y$. 
In fact, by contradiction, suppose that $\im(\alpha)\not\subseteq Y$ and $\im(\alpha)\not\subseteq X\setminus Y$. 
Then there exist $a,b\in X$ such that $a\alpha\not\in Y$ and $b\alpha\not\in X\setminus Y$, 
whence $a\alpha\in X\setminus Y$ and $b\alpha\in Y$.  It follows that  
$b\alpha<a\alpha$ and $(b\alpha)\alpha\ge(a\alpha)\alpha$, i.e. 
$b\alpha<a\alpha$ and $b\alpha\ge a\alpha$, which is a contradiction. 
Thus, as  
$Y\alpha\cup(X\setminus Y)\alpha=\im(\alpha)=\im(\alpha^2)=(\im(\alpha))\alpha$, 
if $\im(\alpha)\subseteq Y$ then $\im(\alpha)=Y\alpha$ and 
if $\im(\alpha)\subseteq X\setminus Y$ then $\im(\alpha)=(X\setminus Y)\alpha$. 

Suppose that $\alpha\not\in\OO(X)$. Then $X\setminus Y\ne\emptyset$, 
whence $Y\alpha\cap(X\setminus Y)\alpha\ne\emptyset$ 
and so, by Proposition \ref{glued}, 
we get $Y\alpha\cap(X\setminus Y)\alpha=\{m\}$, with $m=\min(Y\alpha)=\max((X\setminus Y)\alpha)$.  
Therefore, either $\im(\alpha)=Y\alpha$, and so  $\min\im(\alpha)=m$ and $(X\setminus Y)\alpha=\{m\}$,  
or $\im(\alpha)=(X\setminus Y)\alpha$, and so $\max\im(\alpha)=m$ and $Y\alpha=\{m\}$.

Finally, notice that the second and third conditions are clearly disjoint. 
Since $\alpha\in\OO(X)$ if and only if $Y=X$, then the first condition  is also disjoint from each of the other two.
\end{proof}

\medskip 

Next, we show that $\POP(X)$ is a closed subset of $\PT(X)$. 

\begin{proposition}\label{POPsemigroup}
$\POP(X)$ is a submonoid of $\PT(X)$.
Consequently, $\OP(X)$ is a submonoid of $\T(X)$ (and of $\POP(X)$ and of $\PT(X)$).  
\end{proposition}
\begin{proof}
First, notice that the identity transformation on $X$ is order-preserving, 
whence it belongs to $\POP(X)$ and to $\OP(X)$. 
Since $\OP(X)=\POP(X)\cap\T(X)$, it suffices to show that $\POP(X)$ is closed under composition of mappings. 

Let $\alpha$ and $\beta$ be transformations of $\POP(X)$ with
ideals $Y$ and $Z$, respectively.

We will show that $\alpha\beta\in\POP(X)$, by considering two
(disjoint) cases. In each case, the construction of the ideal of $\alpha\beta$ will be carried out differently.

\smallskip 

\noindent\textsc{case 1.} $Y\alpha\cap\dom(\beta)\subseteq \dom(\beta)\setminus Z$.

Let $W=(Y\cup Z\alpha^{-1})\cap\dom(\alpha\beta)$. 

First, we suppose that $W=\emptyset$. In this case, 
we show that $\alpha\beta$ is order-preserving on
$\dom(\alpha\beta)$. 
Let $x,y\in \dom(\alpha\beta)$ be such that $x\le y$. 
Then $x,y\in \dom(\alpha)\setminus Y$ and $x\alpha,y\alpha\in \dom(\beta)\setminus Z$.
Hence, since $x\le y$, we have $x\alpha\le y\alpha$ and so
$x\alpha\beta\le y\alpha\beta$. 
Thus $\alpha\beta\in\PO(X)$ and so $\alpha\beta\in\POP(X)$. 

Secondly, we assume that $W\neq\emptyset$ and we will prove that $W$ as an ideal of $\alpha\beta$. 

We begin by showing that $\alpha\beta$ is order-preserving on $W$. Let
$x,y\in W$ be such that $x\le y$. If $y\in Y$ then $x\in Y$,
whence $x\alpha\le y\alpha$ and 
$x\alpha, y\alpha\in Y\alpha\cap\dom(\beta)\subseteq\dom(\beta)\setminus Z$, 
from which it follows that $x\alpha\beta\le y\alpha\beta$.
Now, suppose that $y\not\in Y$. If $x\not\in Y$ then $x\alpha\le
y\alpha$ and $x,y\in Z\alpha^{-1}$,
whence $x\alpha, y\alpha\in Z$ and so $x\alpha\beta\le
y\alpha\beta$. On the other hand, if $x\in Y$ then 
$x\alpha\in Y\alpha\cap\dom(\beta)\subseteq\dom(\beta)\setminus Z$ 
and, since $y\alpha\in Z$, it follows that
$x\alpha\beta\le y\alpha\beta$.

Next, we show that $\alpha\beta$ is order-preserving on
$\dom(\alpha\beta)\setminus W$. 
Let $x,y\in \dom(\alpha\beta)\setminus W$ be such that $x\le y$. 
Then $x,y\in \dom(\alpha)\setminus Y$ and $x\alpha,y\alpha\in \dom(\beta)\setminus Z$.
Hence, since $x\le y$, we have $x\alpha\le y\alpha$ and so
$x\alpha\beta\le y\alpha\beta$.

Now, we focus our attention on the condition (OP2). Let $x\in W$
and $y\in \dom(\alpha\beta)\setminus W$. Then, 
as above, we have 
$y\in \dom(\alpha)\setminus Y$ and $y\alpha\in \dom(\beta)\setminus Z$. 

If $x\in Y$ then $x\le y$ and $x\alpha\ge y\alpha$, since $y\in \dom(\alpha)\setminus Y$. 
Moreover, as $x\alpha\in Y\alpha\cap\dom(\beta)\subseteq \dom(\beta)\setminus Z$ 
and $y\alpha\in \dom(\beta)\setminus Z$, we
obtain $x\alpha\beta\ge y\alpha\beta$.

On the other hand, suppose that $x\not\in Y$. 
Then, $x\in Z\alpha^{-1}$ and so $x\alpha\in Z$. 
If $y\le x$ then $y\alpha\le x\alpha$ (since $x,y\in \dom(\alpha)\setminus Y$), 
which implies that $y\alpha\in Z$, a contradiction. 
Hence $x\le y$. Moreover, since $x\alpha\in Z$ and $y\alpha\in \dom(\beta)\setminus Z$, 
we also have $x\alpha\beta\ge y\alpha\beta$.

Therefore, we just proved that $\alpha\beta$ admits $W$ as an ideal and so we have   
$\alpha\beta\in\POP(X)$. 

\smallskip 

\noindent\textsc{case 2.} $Y\alpha\cap\dom(\beta)\not\subseteq \dom(\beta)\setminus Z$.

Notice that, there exists $a\in Y$ such that $a\alpha\in Z\cap\dom(\beta)$. 
Hence, if $x\in \dom(\alpha)\setminus Y$ then $x\alpha\le a\alpha$ and so $x\alpha\in Z$. 
Thus $(\dom(\alpha)\setminus Y)\alpha\subseteq Z$.

Let $V=Y\cap Z\alpha^{-1}\cap\dom(\alpha\beta)$. 

First, we suppose that $V=\emptyset$ and we will show that $\alpha\beta$ is order-preserving on
$\dom(\alpha\beta)$. 
Let $x,y\in \dom(\alpha\beta)$ be such that $x\le y$.
If $x\in \dom(\alpha)\setminus Y$ then $y\in \dom(\alpha)\setminus Y$ and so
$x\alpha,y\alpha\in(\dom(\alpha)\setminus Y)\alpha\subseteq Z$. 
It follows that $x\le y$ implies $x\alpha\le y\alpha$, 
which in turn implies that $x\alpha\beta\le y\alpha\beta$. 
So, suppose that $x\in Y$. 
If $y\in Y$ then, as $x,y\in \dom(\alpha\beta)$, 
we have $x,y\not\in Z\alpha^{-1}$ and so 
$x\alpha,y\alpha\in \dom(\beta)\setminus Z$. 
Hence, in this case, 
$x\le y$ implies $x\alpha\le y\alpha$,
which in turn implies that $x\alpha\beta\le y\alpha\beta$. 
On the other hand, if $y\in \dom(\alpha)\setminus Y$, 
then we have $x\alpha\in\dom(\beta)\setminus Z$ and $y\alpha\in(\dom(\alpha)\setminus Y)\alpha\subseteq Z$, 
from which follows again $x\alpha\beta\le y\alpha\beta$.
Thus $\alpha\beta\in\PO(X)$ and so $\alpha\beta\in\POP(X)$. 

Secondly, we assume that $V\neq\emptyset$ and we will prove that $V$ as an ideal of $\alpha\beta$. 

We begin by showing that $\alpha\beta$ is order-preserving on $V$.
Let $x,y\in V$ be such that $x\le y$. Then $x,y\in Y$ and
$x\alpha,y\alpha\in Z$. It follows that $x\le y$ implies
$x\alpha\le y\alpha$, which in turn implies that $x\alpha\beta\le
y\alpha\beta$.

We continue by showing that $\alpha\beta$ is order-preserving on
$\dom(\alpha\beta)\setminus V$. 
Let $x,y\in \dom(\alpha\beta)\setminus V$ be such that $x\le y$.
If $x\in \dom(\alpha)\setminus Y$ then $y\in \dom(\alpha)\setminus Y$ and so
$x\alpha,y\alpha\in(\dom(\alpha)\setminus Y)\alpha\subseteq Z$. 
It follows that $x\le y$ implies $x\alpha\le y\alpha$, 
which in turn implies that $x\alpha\beta\le y\alpha\beta$. 
So, suppose that $x\in Y$. 
If $y\in Y$ then, as $x,y\in \dom(\alpha\beta)\setminus V$, 
we have $x,y\not\in Z\alpha^{-1}$ and so 
$x\alpha,y\alpha\in \dom(\beta)\setminus Z$. 
Hence, in this case, 
$x\le y$ implies $x\alpha\le y\alpha$,
which in turn implies that $x\alpha\beta\le y\alpha\beta$. 
On the other hand, if $y\in \dom(\alpha)\setminus Y$, 
then we have $x\alpha\in\dom(\beta)\setminus Z$ and $y\alpha\in(\dom(\alpha)\setminus Y)\alpha\subseteq Z$, from which follows again $x\alpha\beta\le y\alpha\beta$.

Next, we prove that $\alpha\beta$ satisfies (OP2) with respect to $V$. 

Let $x\in V$ and $y\in \dom(\alpha\beta)\setminus V$. 
Then $x\in Y$ and $x\alpha\in Z$.

Suppose that $y\in Y$. Then, since $y\in\dom(\alpha\beta)\setminus V$, 
we have $y\not\in Z\alpha^{-1}$ and so $y\alpha\in \dom(\beta)\setminus Z$. 
Hence $x\alpha\beta\ge y\alpha\beta$. On the other hand, if $y\le x$ then $y\alpha\le
x\alpha$ and so $y\alpha\in Z$, a contradiction. Thus, we also
have $x\le y$.

Now, suppose that $y\in \dom(\alpha)\setminus Y$. 
Then $x\le y$ and $x\alpha\ge y\alpha$. Additionally,  
$y\alpha\in(\dom(\alpha)\setminus Y)\alpha\subseteq  Z$ and so 
$x\alpha\beta\ge y\alpha\beta$.

Therefore we proved that $\alpha\beta$ admits $V$ as an ideal,
whence $\alpha\beta\in\POP(X)$, as required.
\end{proof}

\medskip 

Now, let $\alpha\in\POP(X)$, let $Y$ be an ideal of $\alpha$ and $A$ be any subset of $\dom(\alpha)$.
Clearly, if $A\cap Y=\emptyset$ then the restriction of $\alpha$ to $A$ is order-preserving. 
On the other hand, if $A\cap Y\ne\emptyset$ then it is easy to show that $A\cap Y$ is an ideal of the restriction 
of $\alpha$ to $A$ and so this transformation is also orientation-preserving. In short: 

\begin{proposition}\label{rPOP}
Any restriction of an element of $\POP(X)$ also is an element of $\POP(X)$. 
\end{proposition}

Another property of orientation-preserving transformations is the following: 

\begin{proposition}\label{iPOP}
Let $\alpha$ be a partial permutation of $X$. 
If $\alpha$ is orientation-preserving then its inverse function 
$\alpha^{-1}:\im(\alpha)\longrightarrow\dom(\alpha)$ also is orientation-preserving. 
In particular, if $\alpha$ is order-preserving then $\alpha^{-1}$ also is order-preserving. 
\end{proposition}
\begin{proof}
Let $Y$ be an ideal of $\alpha$. 
If $Y=\dom(\alpha)$ then $\alpha$ is order-preserving and it is easy to show that $\alpha^{-1}$ is also order-preserving. On the other hand, if $Y\subsetneq\dom(\alpha)$ then it is a routine matter to prove that 
$(\dom(\alpha)\setminus Y)\alpha$ is an ideal of $\alpha^{-1}$. 
\end{proof}

Denote by $\I(X)$ the submonoid of $\PT(X)$ of all partial permutations. This monoid is called the \textit{symmetric inverse monoid} on $X$. It follows from Proposition \ref{iPOP} that the submonoids $\POI(X)=\PO(X)\cap\I(X)$ and $\POPI(X)=\POP(X)\cap\I(X)$ of $\I(X)$ also are \textit{inverse submonoids} of $\I(X)$. 

\smallskip 

Let us also consider $\S(X)$ as being the \textit{symmetric group} on $X$, i.e. the group of all permutations of $X$. 
Hence, we have the following diagram with respect to the inclusion relation:
\begin{figure}[h] 
\centering \label{fig2}
\begin{picture}(130,115)(0,45)
%
%
{
\put(117.5,107.5){$\bullet$}
\put(124,107.5){\scriptsize $\T(X)$}
\put(67.5,157.5){$\bullet$}
\put(75,157.5){\scriptsize $\PT(X)$}
\put(17.5,107.5){$\bullet$}
\put(0,107.5){\scriptsize $\I(X)$}
\put(120,110){\line(-1,1){50}}
\put(120,90){\line(0,1){20}}
\put(70,140){\line(0,1){20}}
\put(20,110){\line(1,1){50}}
\put(20,90){\line(0,1){20}}
\put(67.5,57.5){$\bullet$}
\put(70,60){\line(1,1){50}}
\put(70,60){\line(-1,1){50}}
\put(70,60){\line(0,-1){20}}
\put(61,72){\scriptsize $\S(X)$}
}
%
%
\put(67.5,37.5){$\bullet$}
\put(70,40){\line(1,1){30}}
\put(70,40){\line(-1,1){30}}
\put(74,37.5){\scriptsize $\OO(X)\cap\S(X)$}
\put(37.5,67.5){$\bullet$}
\put(40,70){\line(1,1){30}}
\put(40,70){\line(-1,1){20}}
\put(5,69){\scriptsize $\POI(X)$}
\put(97.5,67.5){$\bullet$}
\put(100,70){\line(-1,1){30}}
\put(100,70){\line(1,1){20}}
\put(104,67.5){\scriptsize $\OO(X)$}
\put(17.5,87.5){$\bullet$}
\put(20,90){\line(1,1){50}}
\put(-17,87.5){\scriptsize $\POPI(X)$}
\put(117.5,87.5){$\bullet$}
\put(120,90){\line(-1,1){50}}
\put(124,87.5){\scriptsize $\OP(X)$}
\put(67.5,97.5){$\bullet$}
\put(70,100){\line(0,1){40}}
\put(73.5,98){\scriptsize $\PO(X)$}
\put(67.5,137.5){$\bullet$}
\put(93,138){\scriptsize $\POP(X)$}
\end{picture}
\end{figure}

\section{Regularity} \label{reg} 

The following criterion for the regularity of the elements of $\OO(X)$ was proved in 2010 by Mora and Kemprasit. 

\begin{theorem}[{\cite[Theorem 2.4]{Mora&Kemprasit:2010}}] \label{regO}
Let $X$ be a chain and let $\alpha\in\OO(X)$. Then $\alpha$ is a regular element of $\OO(X)$ if and only if the following conditions hold:
\begin{enumerate}
\item If $\im(\alpha)$ has an upper bound in $X$, then $\max\im(\alpha)$ exists; 
\item If $\im(\alpha)$ has a lower bound in $X$, then $\min\im(\alpha)$ exists;
\item If $x\in X\setminus\im(\alpha)$ is neither an upper bound nor a lower bound of $\im(\alpha)$, 
then either $\max\{t\in\im(\alpha)\mid t<x\}$ or $\min\{t\in\im(\alpha)\mid t>x\}$ exists. 
\end{enumerate}
\end{theorem}

Based on this theorem,  Mora and Kemprasit \cite{Mora&Kemprasit:2010} deduced several previous known results. 
For instance, 
that $\OO(\Z)$ is regular while $\OO(\Q)$ and $\OO(\R)$ are not regular, 
being $\Z$, $\Q$ and $\R$ the sets of integers, rational numbers and real numbers, respectively,
with their usual orders. 
See also \cite{Adams&Gould:1989,Aizenstat:1968,Kemprasit&Changphas:2000,Kim&Kozhukhov:2008}. 

\medskip 

The example below shows that $\OP(X)$ may have more regular order-preserving elements than
$\OO(X)$. 

\begin{example}\em
Consider the set $\R$ of real numbers 
equipped with the usual order.
As observed above, $\OO(\R)$ is not regular.

Let $\alpha:\R\longrightarrow\R$ be the mapping defined by
$$
x\alpha=\left\{
\begin{array}{ll}
 -1/x & x\ge1\\
 -1 & x<1~,
\end{array}
\right.
$$
for $x\in\R$. Then, $\alpha\in\OO(\R)$ and $\im(\alpha)=[-1,0[$ and so, 
by Theorem \ref{regO}, $\alpha$ is not regular in $\OO(\R)$.

Now, let $\beta:\R\longrightarrow\R$ be the mapping defined by
$$
x\beta=\left\{
\begin{array}{ll}
 -1/x & -1\le x<0\\
 1 & \mbox{otherwise},
\end{array}
\right.
$$
for $x\in\R$. Then $\beta\in\OP(\R)$, with ideal
$Y=\,]-\infty,0[$. Moreover, we have $\alpha=\alpha\beta\alpha$
and $\beta=\beta\alpha\beta$. Thus $\alpha$ is regular in $\OP(\R)$.
\end{example}

The following property gives us a necessary condition for an order-preserving transformation of $X$ to admit a non order-preserving inverse in $\OP(X)$. As an application, we show that $\OP(\R)$ is not regular. 

\begin{proposition}\label{bound}
If $\alpha\in\OO(X)$ admits an inverse
$\beta\in\OP(X)\setminus\OO(X)$ then $\im(\alpha)$ is bounded.
\end{proposition}
\begin{proof}
First, observe that $\im(\alpha)=\im(\alpha\beta\alpha)\subseteq\im(\beta\alpha)\subseteq\im(\alpha)$, whence 
$\im(\alpha)=\im(\beta\alpha)$. 

\smallskip 

Let $Y$ be an ideal of $\beta$. Then $Y$ is a nonempty proper
subset of $X$, since $\beta\in\OP(X)\setminus\OO(X)$.

We begin by supposing that $(X\setminus Y)\beta\alpha\cap
Y\neq\emptyset$. Then, there exist $a\in Y$ and $b\in X\setminus
Y$ such that $b\beta\alpha=a$. Hence,
$b\beta=b\beta\alpha\beta=a\beta$ and so $m=b\beta=a\beta\in
(X\setminus Y)\beta\cap Y\beta$. By Proposition \ref{glued}, it
follows that $(X\setminus Y)\beta\cap Y\beta=\{m\}$ and
$\max((X\setminus Y)\beta)=m=\min(Y\beta)$. Thus, $\beta$ is
constant in $]-\infty,a]\cup(X\setminus Y)$ with value $m$. In
fact, given $y\in ]-\infty,a]$, we have $y\beta\le a\beta=m$,
whence $y\beta=m$. On the other hand, for $x\in X\setminus Y$, we
have $x\beta\le m$ and so $x\beta\alpha\le
m\alpha=b\beta\alpha=a$. Hence, $x\beta=(x\beta\alpha)\beta=m$. In
addition, since $\beta$ is order-preserving on $Y$, it follow that
$\im(\beta)\subseteq[m,+\infty[$, whence $\im(\beta\alpha)\subseteq
[m\alpha,+\infty[=[a,+\infty[$. If
$x\beta\alpha\in X\setminus Y$, for some $x\in X$, then
$m=(x\beta\alpha)\beta=x\beta$, whence $a=m\alpha=x\beta\alpha\in
X\setminus Y$, which is a contradiction. Therefore
$\im(\beta\alpha)\subseteq Y$ and thus $\im(\beta\alpha)\subseteq
Y\cap[a,+\infty[$. Now, as $Y$ is upper bounded by any element of
$X\setminus Y$, we have $\im(\beta\alpha)\subseteq [a,b[$.

Next, we admit the other case, i.e. $(X\setminus
Y)\beta\alpha\subseteq X\setminus Y$. Then, we also have
$Y\beta\alpha\subseteq X\setminus Y$. In fact, being $y\in Y$ and
$x\in X\setminus Y$, we have $x\beta\le y\beta$ and so
$x\beta\alpha\le y\beta\alpha$. Since $x\beta\alpha\in X\setminus
Y$ and $X\setminus Y$ is an order filter of $X$, then
$y\beta\alpha\in X\setminus Y$, as required. Thus
$\im(\beta\alpha)\subseteq X\setminus Y$. It follows that
$X\beta=(X\beta\alpha)\beta\subseteq (X\setminus Y)\beta$, whence
$Y\beta\cap(X\setminus Y)\beta=Y\beta\cap
X\beta=Y\beta\neq\emptyset$ and so, by Proposition \ref{glued},
there exists $m\in X$ such that $Y\beta=\{m\}$, with
$\max(X\beta)=\max((X\setminus Y)\beta)=m\,(=\min(Y\beta))$. Then
$\im(\beta)\subseteq \,]-\infty, m]$ and so
$\im(\beta\alpha)\subseteq \,]-\infty, m\alpha]$. Now, as
$X\setminus Y$ is lower bounded by any element of $Y$, we have
$\im(\beta\alpha)\subseteq \,]y,m\alpha]$, for any $y\in Y$, which
finishes the proof.
\end{proof}

\begin{example}\em 
Let $\R$ be the set of real numbers equipped with the usual order.
Then $\OP(\R)$ is not a regular semigroup. Furthermore,
it contains non regular order-preserving transformations. 

In fact, let $\alpha:\R\longrightarrow\R$ be the mapping defined by 
$x\alpha=-e^{-x}, $ for $x\in\R$. Then $\alpha\in\OO(\R)$ and
$\im(\alpha)=\,]-\infty,0[$. 
Hence, by Theorem \ref{regO}, 
$\alpha$ is not regular in $\OO(\R)$
and, by Proposition \ref{bound}, $\alpha$ is not regular in $\OP(\R)$ too.
\end{example}

\medskip 

Observe that, as $\POI(X)$ and $\POPI(X)$ are inverse monoids then, in particular, they are regular monoids. 
Next, we show that this is also the case of $\PO(X)$ and $\POP(X)$. 

\smallskip 

Let $\alpha\in\POP(X)\setminus\{\emptyset\}$. Let $Y$ be an ideal of $\alpha$. 
For each $x\in Y\alpha$, choose $z_x\in x\alpha^{-1}\cap Y$. 
For each $x\in\im(\alpha)\setminus(Y\alpha)$, choose $z_x\in x\alpha^{-1}$ (observe that, in this case $z_x\in \dom(\alpha)\setminus Y$). 
Let $D=\{z_x\mid x\in\im(\alpha)\}$. Then 
\begin{equation}\label{zetax}
\begin{array}{rrclcrrcl}
\zeta: & \im(\alpha) &\longrightarrow & D & \quad\text{and}\quad & \alpha|_D: & D & \longrightarrow & \im(\alpha) \\ 
          & x & \longmapsto & z_x && & x & \longmapsto & x\alpha 
\end{array}
\end{equation}
are mutually inverse bijections. By Proposition \ref{rPOP}, $\alpha|_D\in\POP(X)$. In addition, $\alpha|_D$ is a partial permutation of $X$ and so, by Proposition \ref{iPOP}, $\zeta\in\POP(X)$. More precisely, $\alpha|_D,\zeta\in\POPI(X)$. 
Moreover, if $\alpha\in\PO(X)$ then $\alpha|_D,\zeta\in\POI(X)$. 

Let $x\in\dom(\alpha)$. Then $x\alpha\zeta\alpha=(x\alpha)\zeta\alpha=z_{x\alpha}\alpha=x\alpha$. 
Hence $\alpha=\alpha\zeta\alpha$. This proves the following result. 

\begin{theorem}\label{regPOP}
$\PO(X)$ and $\POP(X)$ are regular monoids. 
\end{theorem}

The above construction of $\zeta$ will also play a very important role in proving next theorem. 

\smallskip 

In the case of full orientation-preserving transformations, the situation is much more complex. 
In fact, we have the following criteria for the regularity of the elements of $\OP(X)$: 

\begin{theorem}\label{regOP}
Let $X$ be a chain and let $\alpha\in\OP(X)$. Then $\alpha$ is a regular element of $\OP(X)$ if and only if the following conditions are satisfied: 
\begin{enumerate}
\item If $\im(\alpha)$ has an upper bound or a lower bound in $X$, then $\max\im(\alpha)$ exists or $\min\im(\alpha)$ exists;
\item If $x\in X\setminus\im(\alpha)$ is neither an upper bound nor a lower bound of $\im(\alpha)$, 
then either $\max\{t\in\im(\alpha)\mid t<x\}$ or $\min\{t\in\im(\alpha)\mid t>x\}$ exists. 
\end{enumerate}
\end{theorem}

These criteria are similar to that of full order-preserving transformations (Mora and Kemprasit's Theorem quoted above). However, its proof, 
that we will present in Section \ref{proof}, is much more involved. 
The greater complexity stems significantly from the fact that, in the condition 1 of our result, lower and upper bounds are simultaneously related to minimums and maximums, while in the corresponding conditions 1 and 2 of Mora and Kemprasit's Theorem, lower bounds only relate to minimums and upper bounds with maximums. 

\smallskip 

As an application of Theorem \ref{regOP}, we have the following two examples.  

\begin{example}\em
Consider the set of integers $\Z$ with the usual order. 
Since any non-empty subset of $\Z$ with upper bounds has a maximum and any non-empty subset of $\Z$ with lower bounds has a minimum, 
by Theorem \ref{regOP}, it is easy to deduce that $\OP(\Z)$ is a regular semigroup. 
\end{example} 

\begin{example}\em
Let $\Q$ be the set of rational numbers
equipped with its usual order. 

Let $\alpha:\Q\longrightarrow\Q$ be the mapping defined by
$$
x\alpha=\left\{
\begin{array}{ll}
 1-\frac{1}{x+1} & x\ge0\\
-1-\frac{1}{x-1}  & x<0~,
\end{array}
\right.
$$
for $x\in\Q$. Then, $\alpha\in\OP(\Q)$ (notice that, in fact, we also have that $\alpha\in\OO(\Q)$) and $\im(\alpha)=]-1,1[$. 
Hence, by Theorem \ref{regOP}, $\alpha$ is not regular
in $\OP(\Q)$.

Therefore, the semigroup $\OP(\Q)$ is not regular. 
\end{example}

\medskip 

In \cite{Fernandes&al:2014} Fernandes et al. showed that the product of any two regular elements of $\OO(X)$ is a regular element of $\OO(X)$. 
We finish this section leaving as an open problem the corresponding question for $\OP(X)$. 
\begin{question}
Is the set $\reg(\OP(X))$ of all regular elements of $\OP(X)$ a subsemigroup of $\OP(X)$? 
\end{question} 

\section{The proof of Theorem \ref{regOP}} \label{proof} 

This section consists of a series of properties. All together demonstrate Theorem \ref{regOP}. 

\medskip 

We start with the proof of the direct implication, which is achieved with our first two lemmas. 
With this in mind, let $\alpha$ be a regular element of $\OP(X)$ and let $\beta\in\OP(X)$ be such that $\alpha=\alpha\beta\alpha$. 
Then $\beta\alpha$ is an idempotent of $\OP(X)$ and so, by Proposition \ref{idps}, $\beta\alpha\in\OO(X)$ 
or $\min\im(\beta\alpha)$ exists or $\max\im(\beta\alpha)$ exists. 
If $\beta\alpha\in\OO(X)$ and  $\im(\beta\alpha)$ has an upper bound or a lower bound in $X$ then, 
since $\beta\alpha$ is a regular element of $\OO(X)$, by Theorem \ref{regO}, 
$\max\im(\beta\alpha)$ exists or $\min\im(\beta\alpha)$ exists. 
On the other hand, $\im(\alpha)=\im(\alpha\beta\alpha)\subseteq\im(\beta\alpha)\subseteq\im(\alpha)$ 
and so $\im(\alpha)=\im(\beta\alpha)$. 
Therefore, it follows immediately: 

\begin{lemma} 
Under the above conditions, 
if $\im(\alpha)$ has an upper bound or a lower bound in $X$, then $\max\im(\alpha)$ exists or $\min\im(\alpha)$ exists. 
\end{lemma}

Our next lemma completes the proof of the direct implication of Theorem \ref{regOP}. 
 
\begin{lemma}
Under the above conditions, 
if $x\in X\setminus\im(\alpha)$ is neither an upper bound nor a lower bound of $\im(\alpha)$, 
then either $\max\{t\in\im(\alpha)\mid t<x\}$ or $\min\{t\in\im(\alpha)\mid t>x\}$ exists. 
\end{lemma}
\begin{proof}
Let $x\in X\setminus\im(\alpha)$ be such that $x$ is neither an upper bound nor a lower bound of $\im(\alpha)$. 
As $\im(\alpha)=\im(\beta\alpha)$, then $x\in X\setminus\im(\beta\alpha)$ and 
$x$ is neither an upper bound nor a lower bound of $\im(\beta\alpha)$. 

If $\beta\alpha\in\OO(X)$ then, since $\beta\alpha$ is a regular element of $\OO(X)$, by Theorem \ref{regO}, 
either $\max\{t\in\im(\beta\alpha)\mid t<x\}$ or $\min\{t\in\im(\beta\alpha)\mid t>x\}$ exists, i.e. 
either $\max\{t\in\im(\alpha)\mid t<x\}$ or $\min\{t\in\im(\alpha)\mid t>x\}$ exists. 

Thus, suppose that $\beta\alpha\not\in\OO(X)$ and let $Y$ be the ideal of $\beta\alpha$. 
Notice that $Y$ is a proper subset of $X$. 
Then, 
as $\beta\alpha$ is an idempotent of $\OP(X)$,  
by Proposition \ref{idps}, we have two possible cases. 

\smallskip 

\noindent\textsc{case 1.} $\im(\beta\alpha)=Y\beta\alpha\subseteq Y$. 

Fix $z\in X\setminus Y$ (notice that $X\setminus Y\ne\emptyset$) and 
define a transformation $\gamma$ of $X$ by $x\gamma=x\beta\alpha$, 
if $x\in Y$, and $x\gamma=z$, if $x\in X\setminus Y$.  
Clearly, $\gamma\in\OO(X)$ and $\gamma$ is an idempotent. In particular, $\gamma$ is a regular element of $\OO(X)$. 
Moreover, $\im(\gamma)=Y\beta\alpha\cup\{z\}=\im(\beta\alpha)\cup\{z\}$. 

As $\im(\beta\alpha)\subseteq\im(\gamma)$ and 
$x$ is neither an upper bound nor a lower bound of $\im(\beta\alpha)$, 
then we also have that $x$ is neither an upper bound nor a lower bound of $\im(\gamma)$. 
On the other hand, 
since $z$ is an upper bound of $\im(\gamma)$ (notice that $\im(\beta\alpha)\subseteq Y$ and $z\in X\setminus Y$), 
then $x<z$, whence as also $x\in X\setminus\im(\beta\alpha)$, we obtain $x\in X\setminus\im(\gamma)$. 
Thus, by Theorem \ref{regO}, either $\max\{t\in\im(\gamma)\mid t<x\}$ or $\min\{t\in\im(\gamma)\mid t>x\}$ exists. 

Since $\{t\in\im(\alpha)\mid t<x\}=\{t\in\im(\beta\alpha)\mid t<x\}=\{t\in\im(\gamma)\mid t<x\}$, 
if $\max\{t\in\im(\gamma)\mid t<x\}$ exists then $\max\{t\in\im(\alpha)\mid t<x\}$ also exists. 

On the other hand, since $x<z$, then  
$\{t\in\im(\gamma)\mid t>x\}=\{t\in\im(\beta\alpha)\mid t>x\}\cup\{z\}$. Hence, as $t<z$, for all $t\in\im(\beta\alpha)$, 
if $\min\{t\in\im(\gamma)\mid t>x\}$ exists then $\min\{t\in\im(\beta\alpha)\mid t>x\}=\min\{t\in\im(\alpha)\mid t>x\}$ also exists 
(and coincide with each other). 

Thus, in this case, we showed that either $\max\{t\in\im(\alpha)\mid t<x\}$ or $\min\{t\in\im(\alpha)\mid t>x\}$ exists. 

\smallskip 

\noindent\textsc{case 2.} $\im(\beta\alpha)=(X\setminus Y)\beta\alpha\subseteq X\setminus Y$.

In this case we may proceed similarly to \textsc{case 1}, essentially by switching the roles of $X\setminus Y$ and $Y$. 
\end{proof} 

\medskip 

For the rest of this section, we will prove the converse implication of Theorem \ref{regOP}. 

\smallskip 

We begin by showing Lemmas \ref{a}, \ref{aaa} and \ref{aa}, which are general lemmas, i.e. we are not, for the time being, assuming the conditions 1 and 2 of the theorem. 
However, as will be shown subsequently, these three lemmas are intrinsically related to these two conditions. 

\smallskip 

Thus, let $\alpha$ be any element of $\OP(X)$ and, throughout the rest of this section, let $Y$ be an ideal of $\alpha$. 

\begin{lemma}\label{a}
Let $x\in X$. 
\begin{enumerate}
\item Suppose there exists $b_1\in Y\alpha$ such that $b_1<x$.  
If $d=\max\{t\in\im(\alpha)\mid t<x\}$ or $d=\min\{t\in\im(\alpha)\mid t>x\}$, then $d\in Y\alpha$. 

\item Suppose there exists $b_2\in (X\setminus Y)\alpha$ such that $x<b_2$. 
If $c=\max\{t\in\im(\alpha)\mid t<x\}$ or $c=\min\{t\in\im(\alpha)\mid t>x\}$, then $c\in (X\setminus Y)\alpha$. 

\end{enumerate} 
\end{lemma}
\begin{proof} Recall that, by definition, $x_1\ge x_2$, for all  $x_1\in Y\alpha$ and $x_2\in (X\setminus Y)\alpha$.  
Hence, $b_1\in Y\alpha$, $d\in\im(\alpha)$ and $b_1\le d$ imply $d\in Y\alpha$ and, analogously,  
$b_2\in (X\setminus Y)\alpha$, $c\in\im(\alpha)$ and $c\le b_2$ imply $c\in (X\setminus Y)\alpha$. 

\smallskip 

1. First, suppose that $d=\max\{t\in\im(\alpha)\mid t<x\}$. 
Since $b_1<x$ and $b_1\in\im(\alpha)$, then $b_1\le d$. As $b_1\in Y\alpha$ and $d\in\im(\alpha)$, we get $d\in Y\alpha$. 

Secondly, let $d=\min\{t\in\im(\alpha)\mid t>x\}$. Then, 
we have $b_1<x<d$, whence $b_1<d$. As $d\in\im(\alpha)$ and  $b_1\in Y\alpha$, it follows that $d\in Y\alpha$. 

\smallskip 

2. Suppose that $c=\max\{t\in\im(\alpha)\mid t<x\}$. Then, we have $c<x<b_2$, whence $c<b_2$. 
As $c\in\im(\alpha)$ and $b_2\in (X\setminus Y)\alpha$, we obtain $c\in (X\setminus Y)\alpha$. 

Finally, let $c=\min\{t\in\im(\alpha)\mid t>x\}$.  
Since $b_2>x$ and $b_2\in\im(\alpha)$, then $c\le b_2$. As $b_2\in (X\setminus Y)\alpha$ and $c\in\im(\alpha)$, 
it follows that $c\in (X\setminus Y)\alpha$, as required. 
\end{proof}

\smallskip 

Next, recall the construction (\ref{zetax}) of $\zeta$ given in Section \ref{reg}. For each $x\in Y\alpha$, choose $z_x\in x\alpha^{-1}\cap Y$ and, 
for each $x\in\im(\alpha)\setminus(Y\alpha)$, choose $z_x\in x\alpha^{-1}$ (notice that $z_x\in X\setminus Y$). 
Then $\zeta:\im(\alpha)\longrightarrow X$, $x \longmapsto z_x$, is an injective mapping and $z_x\alpha=x$, for all $x\in\im(\alpha)$. 
Moreover, it is easy to show that $\zeta$ is order-preserving both on $Y\alpha$ and on $(X\setminus Y)\alpha\setminus(Y\alpha)$ and, 
for $x\in Y\alpha$ and $y\in (X\setminus Y)\alpha\setminus(Y\alpha)$, we have $x>y$ and $z_x<z_y$. In other words: 

\begin{lemma}\label{aaa}
The mapping $\zeta:\im(\alpha)\longrightarrow X$ is either order-preserving or $(X\setminus Y)\alpha\setminus(Y\alpha)$ is an ideal of $\zeta$. 
\end{lemma}

This property allows us to to routinely prove the following lemma: 

\begin{lemma}\label{aa}
\begin{enumerate}
\item Suppose that $\alpha\in\OO(X)$. 
\begin{enumerate} 
\item If $\im(\alpha)$ has a maximum then $z_x\le z_{\max\im(\alpha)}$, for all $x\in\im(\alpha)$;  
\item If $\im(\alpha)$ has a minimum then $z_{\min\im(\alpha)}\le z_x$, for all $x\in\im(\alpha)$. 
\end{enumerate}

\item Suppose that $\alpha\not\in\OO(X)$ and $Y\alpha\cap(X\setminus Y)\alpha=\emptyset$.  
\begin{enumerate} 
\item If $\im(\alpha)$ has a maximum then $\max\im(\alpha)=\max(Y\alpha)$ and 
$z_d\le z_{\max\im(\alpha)}<z_c$, for all $d\in Y\alpha$ and for all $c\in(X\setminus Y)\alpha$. 
\item If $\im(\alpha)$ has a minimum then $\min\im(\alpha)=\min((X\setminus Y)\alpha)$ and 
$z_d<z_{\min\im(\alpha)}\le z_c$, for all $d\in Y\alpha$ and for all $c\in(X\setminus Y)\alpha$. 
\end{enumerate}

\item Suppose that $Y\alpha\cap(X\setminus Y)\alpha=\{m\}$, for some $m\in X$.  
\begin{enumerate} 
\item If $\im(\alpha)$ has a maximum then $\max\im(\alpha)=\max(Y\alpha)$, 
$z_d\le z_{\max\im(\alpha)}$, for all $d\in Y\alpha$, and $z_{\max\im(\alpha)}<z_c$, for all $c\in(X\setminus Y)\alpha\setminus\{m\}$. 
\item If $\im(\alpha)$ has a minimum then $\min\im(\alpha)=\min((X\setminus Y)\alpha)$. 
In addition, 
\begin{enumerate} 
\item if $\min\im(\alpha)\neq m$ then $z_d<z_{\min\im(\alpha)}\le z_c$, for all $d\in Y\alpha$ and for all $c\in(X\setminus Y)\alpha\setminus\{m\}$; 
\item if $\min\im(\alpha)=m$ (i.e. $(X\setminus Y)\alpha=\{m\}$) then $z_{\min\im(\alpha)}\le z_x$, for all $x\in\im(\alpha)$. 
\end{enumerate} 
\end{enumerate}
\end{enumerate} 
\end{lemma} 

\medskip 

The three cases considered in the previous lemma will allow us to structure the rest of our proof. 

\medskip 

From now and until the rest of this section, let us suppose that $\alpha$ satisfies conditions 1 and 2 of Theorem \ref{regOP}, i.e. 
\begin{enumerate}
\item if $\im(\alpha)$ has an upper bound or a lower bound in $X$, then $\max\im(\alpha)$ exists or $\min\im(\alpha)$ exists;
\item if $x\in X\setminus\im(\alpha)$ is neither an upper bound nor a lower bound of $\im(\alpha)$, 
then either $\max\{t\in\im(\alpha)\mid t<x\}$ or $\min\{t\in\im(\alpha)\mid t>x\}$ exists. 
\end{enumerate} 
These conditions allow us to extend $\zeta$ to a full transformations on $X$, as follows. We define $\beta\in\T(X)$ by: 
\begin{enumerate}
\item If $x\in\im(\alpha)$, then $x\beta=z_x$; 

\item If $x\in X\setminus\im(\alpha)$ is an upper bound or a lower bound of $\im(\alpha)$, then 
$$
x\beta=\left\{
\begin{array}{ll}
z_{\max\im(\alpha)} & \mbox{if $\im(\alpha)$ has a maximum}\\ 
z_{\min\im(\alpha)} & \mbox{otherwise}
\end{array}\right. 
$$
(notice that if $\max\im(\alpha)$ does not exist then $\min\im(\alpha)$ has to exist, by condition 1 of Theorem \ref{regOP}); 

\item  If $x\in X\setminus\im(\alpha)$ is neither an upper bound nor a lower bound of $\im(\alpha)$, then 
$$
x\beta=\left\{
\begin{array}{ll}
z_{\max\{t\in\im(\alpha)\mid t<x\}} & \mbox{if $\{t\in\im(\alpha)\mid t<x\}$ has a maximum}\\ 
z_{\min\{t\in\im(\alpha)\mid t>x\}} & \mbox{otherwise}
\end{array}\right. 
$$
(notice that if $\max\{t\in\im(\alpha)\mid t<x\}$ does not exist then $\min\{t\in\im(\alpha)\mid t>x\}$ has to exist, by condition 2 of Theorem \ref{regOP}). 
\end{enumerate}

Let $x\in X$. Then $x\alpha\in\im(\alpha)$ and so $x(\alpha\beta\alpha)=((x\alpha)\beta)\alpha=z_{x\alpha}\alpha=x\alpha$. 
Thus, we showed:

\begin{lemma}\label{r}
$\alpha=\alpha\beta\alpha$. 
\end{lemma}

\smallskip 

Our objective now is to prove that $\beta\in\OP(X)$, which, in view of the previous lemma, will complete the proof of the converse implication of Theorem \ref{regOP}. 
This goal will be accomplish by considering, in its turn, each of the cases of Lemma \ref{aa}. 

Observe that, it is a routine matter to show also that $\beta=\beta\alpha\beta$, 
whence $\beta$ is an inverse of $\alpha$ in $\T(X)$ (for the time being, but also in $\OP(X)$, at the end of this section). 

\smallskip 

Another important piece in this process is the notion that we present now. 

Let $C$ be a subset of $X$. We define the \textit{convex closure} of $C$ as being the set 
$$
\cl{C}=\bigcup_{x,y\in C, x\le y}[x,y]. 
$$
It is clear that $\cl{C}$ is a \textit{convex} subset of $X$ (in the sense that, for all $x,y\in\cl{C}$ and for all $z\in X$, $x\le z\le y$ implies $z\in\cl{C}$). 
This concept will help us to construct an ideal for $\beta$. 

\smallskip 

The following lemma will be used throughout various proofs that we will carry out below. 

\begin{lemma}\label{x}
\begin{enumerate}
\item If $x\in \cl{Y\alpha}$ then $x\beta=z_d\in Y$, for some $d\in Y\alpha$. 

\item The transformation $\beta$ is order-preserving on $\cl{Y\alpha}$. 
\end{enumerate}
\end{lemma}
\begin{proof}
1. Let $x\in \cl{Y\alpha}$. If $x\in Y\alpha$ then $x\beta=z_x\in Y$, by definition. Suppose that $x\not\in Y\alpha$. 
Then $b_1<x<b_2$, for some $b_1,b_2\in Y\alpha$, and so $x\in X\setminus\im(\alpha)$ (since $x>b_1$, with $b_1\in Y\alpha$) 
and $x$ is neither an upper bound nor a lower bound of $\im(\alpha)$.  Hence 
$x\beta=z_d$, with 
$$
d=\left\{
\begin{array}{ll}
\max\{t\in\im(\alpha)\mid t<x\} & \mbox{if $\{t\in\im(\alpha)\mid t<x\}$ has a maximum}\\ 
\min\{t\in\im(\alpha)\mid t>x\} & \mbox{otherwise.}
\end{array}\right. 
$$
As $b_1<x$ and $b_1\in Y\alpha$, then $d\in Y\alpha$, by Lemma \ref{a}, and so $z_d\in Y$.  

\smallskip 

2. Let $x,y\in\cl{Y\alpha}$ be such that $x<y$. 

If $x,y\in Y\alpha$, then $x\beta=z_x=x\zeta\le y\zeta=z_y=y\beta$, since $\zeta$ is order-preserving on $Y\alpha$ (we will use   this fact below several times without explicit mention). See Lemma \ref{aaa}. 

Next, suppose that $x\in Y\alpha$ and $y\not\in Y\alpha$. Then $x\beta=z_x$ and $y\beta=z_d$, 
with $d=\max\{t\in\im(\alpha)\mid t<y\}$ or $d=\min\{t\in\im(\alpha)\mid t>y\}$. 
If $d=\max\{t\in\im(\alpha)\mid t<y\}$, as $x<y$, then $x\le d$.  
On the other hand, if $d=\min\{t\in\im(\alpha)\mid t>y\}$ then $d>y>x$, whence $x<d$. 
Thus, in both cases, we have $x\beta=z_x=x\zeta\le d\zeta=z_d=y\beta$. 

Now, we consider the case $x\not\in Y\alpha$ and $y\in Y\alpha$. Then $y\beta=z_y$ and $x\beta=z_d$, 
with  $d=\max\{t\in\im(\alpha)\mid t<x\}$ or $d=\min\{t\in\im(\alpha)\mid t>x\}$. 
If $d=\max\{t\in\im(\alpha)\mid t<x\}$ then $d<x<y$, whence $d<y$. 
On the other hand, if $d=\min\{t\in\im(\alpha)\mid t>x\}$, as $y>x$, then $d\le y$. 
So, in both cases, we get $x\beta=z_d=d\zeta\le y\zeta=z_y=y\beta$. 

Finally, we suppose that $x,y\not\in Y\alpha$. 
Then $x\beta=z_c$, with $c=\max\{t\in\im(\alpha)\mid t<x\}$ or $c=\min\{t\in\im(\alpha)\mid t>x\}$, 
and $y\beta=z_d$, with  $d=\max\{t\in\im(\alpha)\mid t<y\}$ or $d=\min\{t\in\im(\alpha)\mid t>y\}$.  
Let us look at the four possible cases: 
\begin{itemize}
\item If $c=\max\{t\in\im(\alpha)\mid t<x\}$ and $d=\max\{t\in\im(\alpha)\mid t<y\}$ then $c<x<y$, whence $c<y$ and so $c\le d$; 

\item If $c=\max\{t\in\im(\alpha)\mid t<x\}$ and $d=\min\{t\in\im(\alpha)\mid t>y\}$ then $c<x<y<d$ and so $c<d$;  

\item Next, suppose that $c=\min\{t\in\im(\alpha)\mid t>x\}$ and $d=\max\{t\in\im(\alpha)\mid t<y\}$. 
Notice that, in this case, by definition of $\beta$, $\max\{t\in\im(\alpha)\mid t<x\}$ does not exist. 
If $d<x$ then there exists $e\in\im(\alpha)$ such that $d<e<x$, whence $e<y$ and so $e\le d$, 
which is a contradiction. Therefore, $d>x$ and so $c\le d$; 

\item If $c=\min\{t\in\im(\alpha)\mid t>x\}$ and $d=\min\{t\in\im(\alpha)\mid t>y\}$ then $d>y>x$, whence $d>x$ and so $c\le d$. 
\end{itemize}
Thus, in all cases, we have $x\beta=z_c=c\zeta\le d\zeta=z_d=y\beta$, as required. 
\end{proof} 

\medskip

As mentioned above, we will show that $\beta\in\OP(X)$ by considering each of the cases of Lemma \ref{aa}. 
Namely, 
\begin{enumerate}
\item $\alpha\in\OO(X)$, 
\item $\alpha\not\in\OO(X)$ and $Y\alpha\cap(X\setminus Y)\alpha=\emptyset$, and 
\item $Y\alpha\cap(X\setminus Y)\alpha=\{m\}$, for some $m\in X$. 
\end{enumerate}

Notice that, in this last case, we must always have $\alpha\not\in\OO(X)$. 
Moreover, this condition is equivalent to $Y\alpha\cap(X\setminus Y)\alpha\neq\emptyset$, by Proposition \ref{glued}. 
Therefore, the above three cases are all possible cases, which validates our strategy to prove the converse implication of Theorem \ref{regOP}. 

\medskip 

Define 
$$
\lb(\alpha)=\{x\in X\setminus\im(\alpha)\mid\mbox{$x$ is a lower bound of $\im(\alpha)$}\}
$$ 
and 
$$
\ub(\alpha)=\{x\in X\setminus\im(\alpha)\mid\mbox{$x$ is an upper bound of $\im(\alpha)$}\}.
$$

\smallskip 

We start by considering the case where $\alpha\in\OO(X)$. 

\begin{proposition}\label{e1} 
If $\alpha\in\OO(X)$ then $\beta\in\OP(X)$.
\end{proposition}
\begin{proof}
We consider three cases. 

\smallskip 

\noindent{\sc case 1.1}. $\lb(\alpha)\ne\emptyset$ and $\im(\alpha)$ has a maximum. 

Let $Z=\lb(\alpha)$. Then $X\setminus Z=\cl{\im(\alpha)}\cup\ub(\alpha)$. 

By definition, $\beta$ is constant on $Z$. In fact, $x\beta=z_{\max\im(\alpha)}$, for all $x\in\lb(\alpha)\cup\ub(\alpha)$. 
Since $\im(\alpha)$ has a maximum, then combining Lemmas \ref{aa} and \ref{x}, 
we deduce that $\beta$ is order-preserving on $\cl{\im(\alpha)}\cup\ub(\alpha)$. 
On the other hand, if $x\in Z$ and $y\in X\setminus Z$ then, clearly, $x<y$ and, by Lemmas \ref{aa} and \ref{x}, we have $x\beta\ge y\beta$. 
Thus, $Z$ is an ideal of $\beta$ and so $\beta\in\OP(X)$. 

\smallskip 

\noindent{\sc case 1.2}. $\lb(\alpha)\ne\emptyset$ and $\im(\alpha)$ does not have a maximum. 

Let $Z=\lb(\alpha)\cup\cl{\im(\alpha)}$. Then $X\setminus Z=\ub(\alpha)$. 

In this case, $\min\im(\alpha)$ must exist and we have $x\beta=z_{\min\im(\alpha)}$, for all $x\in\lb(\alpha)\cup\ub(\alpha)$. 
Then, by using Lemmas \ref{aa} and \ref{x}, it is easy to conclude that $Z$ is an ideal of $\beta$ and so $\beta\in\OP(X)$. 

\smallskip 

\noindent{\sc case 1.3}. $\lb(\alpha)=\emptyset$. 

If $\im(\alpha)$ has a maximum then, as in {\sc case 1.1}, by Lemmas \ref{aa} and \ref{x}, 
we obtain that $\beta$ is order-preserving on $\cl{\im(\alpha)}\cup\ub(\alpha)$, i.e. $\beta\in\OO(X)$ and so $\beta\in\OP(X)$. 
On the other hand, if $\im(\alpha)$ does not have a maximum, again by Lemmas \ref{aa} and \ref{x}, 
it is easy to conclude that $Z=\cl{\im(\alpha)}$ is an ideal of $\beta$ (notice that $\ub(\alpha)=\emptyset$ or $\im(\alpha)$ has a minimum) and so, 
also in this case, $\beta\in\OP(X)$, as required. 
\end{proof} 

\medskip 

Now, we move on for the two cases where $\alpha\not\in\OO(X)$. 
Before studying them separately, we present the following lemma, for the convenience of its proof. 

\begin{lemma}\label{c}
\begin{enumerate}
\item If $Y\alpha\cap(X\setminus Y)\alpha=\emptyset$, then $\beta$ is order-preserving on $W=\cl{(X\setminus Y)\alpha}$. 

\item If $Y\alpha\cap(X\setminus Y)\alpha=\{m\}$, for some $m\in X$,  then $\beta$ is order-preserving on $W=\cl{(X\setminus Y)\alpha\setminus\{m\}}$. 

\item[] Moreover, in both cases, if $x\in W$ then $x\beta=z_c\in X\setminus Y$, for some $c\in (X\setminus Y)\alpha\setminus (Y\alpha)$. 

\end{enumerate} 
\end{lemma}
\begin{proof}
First, notice that, in the case of 2, we have $\max((X\setminus Y)\alpha)=m=\min(Y\alpha)$ and so $x<m$, for all $x\in W$. 
In particular, $m\not\in W$. 

Let $x\in W$. 
If $x\in (X\setminus Y)\alpha$ then $x\not\in Y\alpha$ and so $x\beta=z_x\in X\setminus Y$, by definition. 
Suppose that $x\not\in(X\setminus Y)\alpha$. 
Then $b_1<x<b_2$, for some $b_1,b_2\in (X\setminus Y)\alpha$ (and $b_2<m$, in the case of 2), and so $x\in X\setminus\im(\alpha)$ 
(since $x<b_2$, with $b_2\in (X\setminus Y)\alpha$) 
and $x$ is neither an upper bound nor a lower bound of $\im(\alpha)$.  Hence 
$x\beta=z_c$, with 
$$
c=\left\{
\begin{array}{ll}
\max\{t\in\im(\alpha)\mid t<x\} & \mbox{if $\{t\in\im(\alpha)\mid t<x\}$ has a maximum}\\ 
\min\{t\in\im(\alpha)\mid t>x\} & \mbox{otherwise.}
\end{array}\right. 
$$
As $x<b_2$ and $b_2\in (X\setminus Y)\alpha$, then $c\in(X\setminus Y)\alpha$, by Lemma \ref{a}. In the case of 1, it is immediate that $c\not\in Y\alpha$. In the case of 2, if $c\in Y\alpha$ then $c=m$. In this case, if $m=\max\{t\in\im(\alpha)\mid t<x\}$ then $m<x<b_2<m$, which is a contradiction, and if $m=\min\{t\in\im(\alpha)\mid t>x\}$, as $b_2>x$, then $m<b_2$, which is again a contradiction. 
Hence, we also have $c\not\in Y\alpha$. Thus, in both cases, $c\in(X\setminus Y)\alpha\setminus(Y\alpha)$ and so $z_c\in X\setminus Y$.  

Let $x,y\in W$ be such that $x<y$. By using the fact that $\zeta$ is also order-preserving on $(X\setminus Y)\alpha\setminus (Y\alpha)$ and replacing,  in the corresponding part of the proof of Lemma \ref{x}, each instance of $Y\alpha$ by $(X\setminus Y)\alpha\setminus (Y\alpha)$, 
we show that $x\beta\le y\beta$, as required. 
\end{proof} 

Notice that, if $\alpha\in\OO(X)$ then the previous lemma is trivially true. However, it has no interest in this case. 
We will only use it in the case where $\alpha\not\in\OO(X)$. 

\medskip 

Now, admit that $Y\alpha\cap(X\setminus Y)\alpha=\emptyset$, $X\setminus Y\neq\emptyset$ 
and there exists $x\in X$ such that $c<x<d$, for all $c\in(X\setminus Y)\alpha$ and for all $d\in Y\alpha$. 
Since $\{t\in\im(\alpha)\mid t<x\}=(X\setminus Y)\alpha$ and $\{t\in\im(\alpha)\mid t>x\}=Y\alpha$ (for any such $x$), then 
$(X\setminus Y)\alpha$ has a maximum or $Y\alpha$ has a minimum. Hence, we clearly have: 

\begin{lemma}\label{bb}
If $Y\alpha\cap(X\setminus Y)\alpha=\emptyset$ and  $X\setminus Y\ne\emptyset$, then 
$$
x\beta=\left\{
\begin{array}{lll}
z_{\max((X\setminus Y)\alpha)} &\in X\setminus Y& \mbox{if $(X\setminus Y)\alpha$ has a maximum}\\ 
z_{\min(Y\alpha)} &\in Y & \mbox{otherwise}, 
\end{array}\right. 
$$
for $x\in X$ such that $c<x<d$, for all $c\in(X\setminus Y)\alpha$ and for all $d\in Y\alpha$. 
\end{lemma}

Next,  in the second case of Lemma \ref{aa}, we prove that $\beta\in\OP(X)$. 

\begin{proposition}\label{e} 
If $\alpha\not\in\OO(X)$ and $Y\alpha\cap(X\setminus Y)\alpha=\emptyset$ then $\beta\in\OP(X)$.
\end{proposition}
\begin{proof} Observe that, as $\alpha\not\in\OO(X)$, then $X\setminus Y\ne\emptyset$. 

\smallskip 

We begin by proving that $\beta$ is order-preserving on $\lb(\alpha)\cup\cl{(X\setminus Y)\alpha}$. 

Let $x,y\in \lb(\alpha)\cup\cl{(X\setminus Y)\alpha}$ be such that $x<y$. 
Since $\beta$ is constant on $\lb(\alpha)$ and, by Lemma \ref{c}, $\beta$ is order-preserving on $\cl{(X\setminus Y)\alpha}$, 
it suffices to consider $x\in\lb(\alpha)$ and $y\in\cl{(X\setminus Y)\alpha}$. 
In this case, we have $x\beta=z_{\max\im(\alpha)}$ or $x\beta=z_{\min\im(\alpha)}$ and, by Lemma \ref{c}, $y\beta=z_c$, 
for some $c\in (X\setminus Y)\alpha$. 
Then, by Lemma \ref{aa}, it follows that $x\beta\le z_c=y\beta$. 

\smallskip 

We continue by showing that $\beta$ is order-preserving on $\cl{Y\alpha}\cup\ub(\alpha)$. 

Let $x,y\in \cl{Y\alpha}\cup\ub(\alpha)$ be such that $x<y$. Once again, 
by Lemma \ref{x}, we have that $\beta$ is order-preserving on $\cl{Y\alpha}$ and, by definition, 
$\beta$ is constant on $\ub(\alpha)$. So, it suffices to consider $x\in \cl{Y\alpha}$ and $y\in\ub(\alpha)$. 
In this case, we have $y\beta=z_{\max\im(\alpha)}$ or $y\beta=z_{\min\im(\alpha)}$ and, by Lemma \ref{x}, $x\beta=z_d$, 
for some $d\in Y\alpha$. Then, by Lemma \ref{aa}, it follows that $x\beta=z_d\le y\beta$. 

\smallskip 

Next, take $x\in \lb(\alpha)\cup\cl{(X\setminus Y)\alpha}$ and $y\in \cl{Y\alpha}\cup\ub(\alpha)$. Clearly, $x<y$. 
Let us prove that we also have $x\beta\ge y\beta$. 
If $x\in\lb(\alpha)$ and $y\in\cl{Y\alpha}$ then $x\beta=z_{\max\im(\alpha)}$ or $x\beta=z_{\min\im(\alpha)}$ and, by Lemma \ref{x}, $y\beta=z_d$, 
for some $d\in Y\alpha$, whence $x\beta\ge z_d=y\beta$, by Lemma \ref{aa}. 
If $x\in\lb(\alpha)$ and $y\in \ub(\alpha)$ then $x\beta=y\beta$. 
If $x\in\cl{(X\setminus Y)\alpha}$ and $y\in\cl{Y\alpha}$ then, by Lemmas \ref{c} and \ref{x}, 
$x\beta=z_c\in X\setminus Y$, for some $c\in (X\setminus Y)\alpha$, and $y\beta=z_d\in Y$, for some $d\in Y\alpha$, 
whence $x\beta=z_c>z_d=y\beta$. 
Finally, if  $x\in\cl{(X\setminus Y)\alpha}$ and $y\in \ub(\alpha)$ then, by Lemma \ref{c}, $x\beta=z_c$, 
for some $c\in (X\setminus Y)\alpha$, and $y\beta=z_{\max\im(\alpha)}$ or $y\beta=z_{\min\im(\alpha)}$, 
whence $x\beta=z_c\ge y\beta$, by Lemma \ref{aa}. 

\smallskip 

Now, let $B=\{x\in X\mid\mbox{$c<x<d$, for all $c\in(X\setminus Y)\alpha$ and for all $d\in Y\alpha$}\}$. 
We will consider three cases. 

\smallskip 

\noindent{\sc case 2.1}. $B=\emptyset$. 

In this case, being $Z=\lb(\alpha)\cup\cl{(X\setminus Y)\alpha}$, we just proved that $Z$ is an ideal of $\beta$ 
(notice that $X\setminus Z= \cl{Y\alpha}\cup\ub(\alpha)$) and so $\beta\in\OP(X)$. 

\smallskip 

\noindent{\sc case 2.2}. $B\ne\emptyset$ and $(X\setminus Y)\alpha$ has a maximum. 

Notice that, for $x\in B$, we have $x\beta=z_{\max((X\setminus Y)\alpha)}\in X\setminus Y$, by Lemma \ref{bb}. 

Take $Z=\lb(\alpha)\cup\cl{(X\setminus Y)\alpha}\cup B$ and let us prove that $Z$ is an ideal of $\beta$. 

In order to prove that $\beta$ is order-preserving on $Z$, let $x,y\in Z$ be such that $x<y$.
Since we already proved that $\beta$ is order-preserving on $\lb(\alpha)\cup\cl{(X\setminus Y)\alpha}$,  
it suffices to consider $x\in \lb(\alpha)\cup\cl{(X\setminus Y)\alpha}$ and $y\in B$. 
Then $x\beta=z_{\max\im(\alpha)}$ or $x\beta=z_{\min\im(\alpha)}$ or, by Lemma \ref{c}, 
$x\beta=z_c\in X\setminus Y$, for some $c\in (X\setminus Y)\alpha$, and 
$y\beta=z_{\max((X\setminus Y)\alpha)}\in X\setminus Y$. 
If  $x\beta=z_{\max\im(\alpha)}$ or $x\beta=z_{\min\im(\alpha)}$ then, by Lemma \ref{aa}, 
we have $x\beta\le z_{\max((X\setminus Y)\alpha)}=y\beta$. 
On the other hand, if $x\beta=z_c$, for some $c\in (X\setminus Y)\alpha$, then $c\le \max((X\setminus Y)\alpha)$ and so, 
as  $\zeta$ is order-preserving on $(X\setminus Y)\alpha$, we have 
$x\beta=z_c=c\zeta\le (\max((X\setminus Y)\alpha))\zeta=z_{\max((X\setminus Y)\alpha)}=y\beta$. 
Thus, we proved that $\beta$ is order-preserving on $Z$. 

Since $X\setminus Z=\cl{Y\alpha}\cup\ub(\alpha)$, we have already proved that $\beta$ is order-preserving on $X\setminus Z$. 

Let $x\in Z$ and $y\in X\setminus Z$. Clearly, $x<y$. If $x\in \lb(\alpha)\cup\cl{(X\setminus Y)\alpha}$, 
we have proved above that $x\beta\ge y\beta$. So, suppose that $x\in B$. Then 
$x\beta=z_{\max((X\setminus Y)\alpha)}\in X\setminus Y$. 
If $y\in \ub(\alpha)$ then $y\beta=z_{\max\im(\alpha)}$ or $y\beta=z_{\min\im(\alpha)}$ and so, by Lemma \ref{aa}, 
we have $y\beta\le z_{\max((X\setminus Y)\alpha)}=x\beta$. 
On the other hand, if $y\in\cl{Y\alpha}$ then, by Lemma \ref{x}, $y\beta=z_d\in Y$, for some $d\in Y\alpha$, and so 
$y\beta=z_d < z_{\max((X\setminus Y)\alpha)}=x\beta$. 

Thus, we have proved that $Z$ is an ideal of $\beta$ and so $\beta\in\OP(X)$. 

\smallskip 

\noindent{\sc case 2.3}. $B\ne\emptyset$ and $(X\setminus Y)\alpha$ does not have a maximum. 

In this case, $\min(Y\alpha)$ exists and, for $x\in B$, by Lemma \ref{bb}, we have $x\beta=z_{\min(Y\alpha)} \in Y$.

Let $Z=\lb(\alpha)\cup\cl{(X\setminus Y)\alpha}$. We aim to show that $Z$ is an ideal of $\beta$. 

We already proved that $\beta$ is order-preserving on $Z$. 

We proceed by showing that $\beta$ is order-preserving on $X\setminus Z=B\cup\cl{Y\alpha}\cup\ub(\alpha)$. 
Let $x,y\in X\setminus Z$ be such that $x<y$. 
Since we already proved that $\beta$ is order-preserving on $\cl{Y\alpha}\cup\ub(\alpha)$, 
it suffices to consider $x\in B$ and $y\in \cl{Y\alpha}\cup\ub(\alpha)$. 
If $y\in \ub(\alpha)$ then $y\beta=z_{\max\im(\alpha)}$ or $y\beta=z_{\min\im(\alpha)}$ and so, by Lemma \ref{aa}, 
we have $x\beta=z_{\min(Y\alpha)} \le y\beta$. 
On the other hand, if $y\in \cl{Y\alpha}$ then $y\beta=z_d$, for some $d\in Y\alpha$, 
whence $\min(Y\alpha)\le d$ and so, as $\zeta$ is order-preserving on $Y\alpha$, we have  
$x\beta= z_{\min(Y\alpha)} =(\min(Y\alpha))\zeta\le d\zeta=z_d=y\beta$. 
Hence, $\beta$ is order-preserving on $X\setminus Z$.

Finally, let $x\in Z$ and $y\in X\setminus Z$. It is clear that $x<y$. 
If $y\in \cl{Y\alpha}\cup\ub(\alpha)$ then we already proved that $x\beta\ge y\beta$. 
So, let us suppose that $y\in B$. Then $y\beta=z_{\min(Y\alpha)} \in Y$. 
If $x\in \lb(\alpha)$ then $x\beta=z_{\max\im(\alpha)}$ or $x\beta=z_{\min\im(\alpha)}$ and so, by Lemma \ref{aa}, 
we have $y\beta=z_{\min(Y\alpha)} \le x\beta$. 
On the other hand, if $x\in\cl{(X\setminus Y)\alpha}$ then, by Lemma \ref{c}, we have 
$x\beta=z_c\in X\setminus Y$, for some $c\in (X\setminus Y)\alpha$,  and so 
$y\beta=z_{\min(Y\alpha)} < z_c=x\beta$. 

Thus, also in this case, we have proved that $Z$ is an ideal of $\beta$ and so $\beta\in\OP(X)$, as required. 
\end{proof} 

\medskip 

Finally, we consider the third and last case of Lemma \ref{aa}, i.e. $Y\alpha\cap(X\setminus Y)\alpha=\{m\}$, for some $m\in X$, or 
equivalently $Y\alpha\cap(X\setminus Y)\alpha\neq\emptyset$ (recall that, in this case, we also have $\alpha\not\in\OO(X)$). 

\begin{lemma}\label{b}
If $Y\alpha\cap(X\setminus Y)\alpha=\{m\}$, for some $m\in X$,  
and $(X\setminus Y)\alpha\setminus\{m\}\ne\emptyset$ then 
$$
x\beta=\left\{
\begin{array}{lll}
z_{\max((X\setminus Y)\alpha\setminus\{m\})} &\in X\setminus Y& \mbox{if $(X\setminus Y)\alpha\setminus\{m\}$ has a maximum}\\ 
z_{m} &\in Y & \mbox{otherwise}, 
\end{array}\right. 
$$
for $x\in X$ such that $c<x<m$, for all $c\in(X\setminus Y)\alpha\setminus\{m\}$. 
\end{lemma}
\begin{proof}
Let $x\in X$ be such that $c<x<m$, for all $c\in(X\setminus Y)\alpha\setminus\{m\}$. 
Then $\{t\in\im(\alpha)\mid t<x\}=(X\setminus Y)\alpha\setminus\{m\}$ and $m=\min\{t\in\im(\alpha)\mid t>x\}$ and so the result follows. 
\end{proof} 

With the following proposition we conclude the proof of Theorem \ref{regOP}. 

\begin{proposition}\label{d} 
If $Y\alpha\cap(X\setminus Y)\alpha\neq\emptyset$ then $\beta\in\OP(X)$. 
\end{proposition}
\begin{proof}
Let $m\in X$ be such that $Y\alpha\cap(X\setminus Y)\alpha=\{m\}$. 
Recall that $\max((X\setminus Y)\alpha)=m=\min(Y\alpha)$. 

\smallskip 

Firstly, we prove that $\beta$ is order-preserving on $\lb(\alpha)\cup\cl{(X\setminus Y)\alpha\setminus\{m\}}$. 

Let $x,y \in \lb(\alpha)\cup\cl{(X\setminus Y)\alpha\setminus\{m\}}$ be such that $x<y$. As $\beta$ is constant in $\lb(\alpha)$, by definition, and order-preserving on $\cl{(X\setminus Y)\alpha\setminus\{m\}}$, by Lemma \ref{c}, it suffices to consider the case $x\in\lb(\alpha)$ and $y\in \cl{(X\setminus Y)\alpha\setminus\{m\}}$. 
In this case, $x\beta=z_{\max\im(\alpha)}$ or $x\beta=z_{\min\im(\alpha)}$ and, by Lemma \ref{c}, $y\beta=z_c$, 
for some $c\in (X\setminus Y)\alpha\setminus\{m\}$. 
Then, by Lemma \ref{aa}, $x\beta\le z_c=y\beta$. 

\smallskip 

In second place, 
if $(X\setminus Y)\alpha\ne\{m\}$ or $\im(\alpha)$ has a maximum, 
we prove that $\beta$ is order-preserving on $\cl{Y\alpha}\cup\ub(\alpha)$. 

Let $x,y\in \cl{Y\alpha}\cup\ub(\alpha)$ be such that $x<y$. As $\beta$ is constant in $\ub(\alpha)$, by definition, and order-preserving on $\cl{Y\alpha}$, by Lemma \ref{x}, it suffices to consider the case $x\in\cl{Y\alpha}$ and $y\in \ub(\alpha)$. 
In this case, $x\beta=z_d$, for some $d\in Y\alpha$, by Lemma \ref{x}, and $y\beta=z_{\max\im(\alpha)}$ or $y\beta=z_{\min\im(\alpha)}$. Then, by Lemma \ref{aa}, $x\beta=z_d\le y\beta$. 

\smallskip 

Thirdly, take $x\in \lb(\alpha)\cup\cl{(X\setminus Y)\alpha\setminus\{m\}}$ and $y\in \cl{Y\alpha}\cup\ub(\alpha)$. 
It is clear that $x<y$. Moreover, 
if $(X\setminus Y)\alpha\ne\{m\}$ or $\im(\alpha)$ has a maximum, 
we also have $x\beta\ge y\beta$. In fact, if $x\in\lb(\alpha)$ and $y\in\ub(\alpha)$ then $x\beta=y\beta$. 
If  $x\in\lb(\alpha)$ and $y\in \cl{Y\alpha}$ then 
$x\beta=z_{\max\im(\alpha)}$ or $x\beta=z_{\min\im(\alpha)}$ and $y\beta=z_d$, 
for some $d\in Y\alpha$, by Lemma \ref{x}, 
whence $x\beta\ge z_d=y\beta$, by Lemma \ref{aa}. 
If $x\in \cl{(X\setminus Y)\alpha\setminus\{m\}}$ and $y\in\ub(\alpha)$ then $x\beta=z_c$, 
for some $c\in (X\setminus Y)\alpha\setminus\{m\}$, by Lemma \ref{c}, and $y\beta=z_{\max\im(\alpha)}$ or $y\beta=z_{\min\im(\alpha)}$, 
whence $x\beta=z_c\ge y\beta$, by Lemma \ref{aa}. 
Finally, if $x\in \cl{(X\setminus Y)\alpha\setminus\{m\}}$ and $y\in \cl{Y\alpha}$ then, by Lemmas \ref{c} and \ref{x}, 
$x\beta=z_c\in X\setminus Y$, for some $c\in (X\setminus Y)\alpha\setminus\{m\}$, and  
$y\beta=z_d\in Y$, for some $d\in Y\alpha$, whence $x\beta=z_c>z_d=y\beta$. 

\smallskip 

Next, we will consider four cases. 

Let $A=\{x\in X\mid\mbox{$t<x<m$, for all $t\in(X\setminus Y)\alpha\setminus\{m\}$}\}$.

\smallskip 

\noindent{\sc case 1}. $(X\setminus Y)\alpha\setminus\{m\}$ has a maximum. 

Let $a=\max((X\setminus Y)\alpha\setminus\{m\})$. 

Let $Z=]-\infty,m[$. Notice that, $Z=\lb(\alpha)\cup\cl{(X\setminus Y)\alpha\setminus\{m\}}\cup A$ (and, 
in this case, $A=]a,m[$) and $X\setminus Z=\cl{Y\alpha}\cup\ub(\alpha)$. 

If $y\in A$ then $y\beta=z_a\in X\setminus Y$, by Lemma \ref{b}, and so $\beta$ is constant in $A$. 
Let $x,y\in Z$ be such that $x<y$. As $\beta$ is constant in $A$ and order-preserving on $\lb(\alpha)\cup\cl{(X\setminus Y)\alpha\setminus\{m\}}$, in order to prove that $\beta$ is order-preserving on $Z$, it suffices to consider the case 
$x\in \lb(\alpha)\cup\cl{(X\setminus Y)\alpha\setminus\{m\}}$ and $y\in A$. 
Then $x\beta=z_{\max\im(\alpha)}$ or $x\beta=z_{\min\im(\alpha)}$ or $x\beta=z_c$, 
for some $c\in (X\setminus Y)\alpha\setminus\{m\}$. 
If $x\beta=z_{\max\im(\alpha)}$ or $x\beta=z_{\min\im(\alpha)}$ then, by Lemma \ref{aa}, $x\beta\le z_a=y\beta$. 
If $x\beta=z_c$, 
for some $c\in (X\setminus Y)\alpha\setminus\{m\}$, then $c\le a$ and so $x\beta=z_c\le z_a=y\beta$, by Lemma \ref{c}. 
Hence, we proved that $\beta$ is order-preserving on $Z$. 

Above, we already proved that $\beta$ is also order-preserving on $X\setminus Z$. 

Let $x\in Z$ and $y\in X\setminus Z$. Clearly, $x<y$. If $x\in \lb(\alpha)\cup\cl{(X\setminus Y)\alpha\setminus\{m\}}$ then, as proved above, we have $x\beta\ge y\beta$. So, suppose that $x\in A$. 
If $y\in \cl{Y\alpha}$ then $y\beta=z_d\in Y$, 
for some $d\in Y\alpha$, by Lemma \ref{x}, whence $x\beta=z_a>z_d=y\beta$. 
If $y\in\ub(\alpha)$ then $y\beta=z_{\max\im(\alpha)}$ or $y\beta=z_{\min\im(\alpha)}$ and so, 
by Lemma \ref{aa}, we have $x\beta=z_a\ge y\beta$. 

Thus, in this case, we showed that $Z$ is an ideal of $\beta$ and so $\beta\in\OP(X)$. 

\smallskip 

\noindent{\sc case 2}. $(X\setminus Y)\alpha\setminus\{m\}\ne\emptyset$ and does not have a maximum. 

Let $Z=\lb(\alpha)\cup\cl{(X\setminus Y)\alpha\setminus\{m\}}$. Then $X\setminus Z=A\cup\cl{Y\alpha}\cup\ub(\alpha)$.  

We proved above that $\beta$ is order-preserving on $Z$. 

If $x\in A$ then $x\beta=z_m\in Y$, by Lemma \ref{b}, and so $\beta$ is constant in $A$. 
Let $x,y\in X\setminus Z$ be such that $x<y$. As $\beta$ is constant in $A$ 
and order-preserving on $\cl{Y\alpha}\cup\ub(\alpha)$, 
in order to prove that $\beta$ is order-preserving on $X\setminus Z$, it suffices to consider the case 
$x\in A$ and $y\in \cl{Y\alpha}\cup\ub(\alpha)$. 
If $y\in \cl{Y\alpha}$ then $y\beta=z_d$, for some $d\in Y\alpha$, by Lemma \ref{x}, 
and so, as $m\le d$, we have $x\beta=z_m\le z_d=y\beta$. 
If $y\in\ub(\alpha)$ then $y\beta=z_{\max\im(\alpha)}$ or $y\beta=z_{\min\im(\alpha)}$, 
whence $x\beta=z_m\le y\beta$, by Lemma \ref{aa}. 
Thus, $\beta$ is order-preserving on $X\setminus Z$. 

Let $x\in Z$ and $y\in X\setminus Z$. Clearly, $x<y$.  
If $y\in \cl{Y\alpha}\cup\ub(\alpha)$ then, as proved above, we have $x\beta\ge y\beta$. So, suppose that $y\in A$. 
If $x\in\lb(\alpha)$ then $x\beta=z_{\max\im(\alpha)}$ or $x\beta=z_{\min\im(\alpha)}$ and so, 
by Lemma \ref{aa}, we have $x\beta\ge z_m=y\beta$. 
If $x\in \cl{(X\setminus Y)\alpha\setminus\{m\}}$ then $x\beta=z_c\in X\setminus Y$, for some $c\in (X\setminus Y)\alpha\setminus\{m\}$, by Lemma \ref{c}, whence $x\beta=z_c>z_m=y\beta$. 

Thus, also in this case, we showed that $Z$ is an ideal of $\beta$ and so $\beta\in\OP(X)$.  

\smallskip

\noindent{\sc case 3}. $(X\setminus Y)\alpha=\{m\}$ and $\im(\alpha)$ has a maximum. 

Let $Z=\lb(\alpha)$. Then $X\setminus Z=\cl{Y\alpha}\cup\ub(\alpha)$. 

In this case, if $Z=\emptyset$ then, as proved above, $\beta$ is order-preserving on $\cl{Y\alpha}\cup\ub(\alpha)=X$, 
whence $\beta\in\OO(X)$ and so $\beta\in\OP(X)$.
On the other hand, if $Z\ne\emptyset$ then, as proved above, $Z$ is an ideal of $\beta$ and so again we obtain 
$\beta\in\OP(X)$.

\smallskip

\noindent{\sc case 4}. $(X\setminus Y)\alpha=\{m\}$ and $\im(\alpha)$ does not have a maximum. 

Let $Z=\lb(\alpha)\cup\cl{Y\alpha}$. Then $X\setminus Z=\ub(\alpha)$. 

Notice that, in this case, we have $\min\im(\alpha)=m$ and $x\beta=z_m$, for all $x\in\lb(\alpha)\cup\ub(\alpha)$. 
Moreover, by Lemma \ref{aa}, we have $z_m\le z_x$, for all $x\in\im(\alpha)$ and, by Lemma \ref{x}, $\beta$ is order-preserving on $\cl{Y\alpha}$. All this together allow us to easily deduce, in this last case, that $Z$ is an ideal of $\beta$ 
and thus $\beta\in\OP(X)$, as required. 
\end{proof} 


\section{Green's relations}\label{RGR}

In this section we characterize Green's relations in $\OP(X)$, $\PO(X)$, $\POP(X)$, $\POI(X)$ and $\POPI(X)$. 
We begin by focusing our attention in the semigroup $\OP(X)$, which is, in fact, the most challenging case. 

\smallskip 

First, recall the following description of the Green's relations in $\T(X)$. 
Let $\alpha,\beta\in\T(X)$. Then, in $\T(X)$, we have \cite{Howie:1995}: 
\begin{itemize}
\item $\alpha\GRL\beta$ if and only if $\im(\alpha)=\im(\beta)$;
\item $\alpha\GRR\beta$ if and only if $\ker(\alpha)=\ker(\beta)$;
\item $\alpha\GRD\beta$ if and only if $|\im(\alpha)|=|\im(\beta)|$;
\item $\GRJ=\GRD$. 
\end{itemize}

A description of the Green's relations in $\OO(X)$ was given by Fernandes et al. in \cite{Fernandes&al:2014}.  
For $\OP(X)$, we will find characterizations very similar to those of $\OO(X)$. The complexity of their demonstrations is in line with those of $\OO(X)$, 
with the exception of the description of Green's relation $\GRL$, whose proof is surprisingly much harder.

\medskip 

Let $\alpha,\beta\in\OP(X)$. If $\alpha\GRL\beta$ in $\OP(X)$ then $\alpha\GRL\beta$ in $\T(X)$
and so $\im(\alpha)=\im(\beta)$.  Next, we aim to show that the converse is also true. 
We start by proving 
two lemmas. 

First, notice that, if $\alpha\in\POP(X)$ admits $Y$ as an ideal then $x\leq y$,  
for all $x\in (X\setminus Y)\alpha$ and $y\in Y\alpha$. 

\begin{lemma}\label{lema1L}
Let $\alpha,\beta\in\OP(X)$ be such that $\im(\alpha)=\im(\beta)$. 
If $A$ and $B$ are ideals of $\alpha$ and $\beta$, respectively, 
then $A\alpha\subseteq B\beta$ or $B\beta\subseteq A\alpha$. 
\end{lemma}
\begin{proof}
Suppose that $A\alpha\nsubseteq B\beta$. 
Then, there exists $z\in A\alpha$ such that $z\not\in B\beta$. 
As $z\in A\alpha\subseteq\im(\alpha)=\im(\beta)=B\beta\cup (X\setminus B)\beta$ 
and $z\not\in B\beta$, then $z\in(X\setminus B)\beta$. It follows that 
$z<y$, for all $y\in B\beta$. 

Let $y\in B\beta$. Since $B\beta\subseteq\im(\beta)=\im(\alpha)=A\alpha\cup (X\setminus A)\alpha$, 
then $y\in A\alpha$ or $y\in(X\setminus A)\alpha$. 
If $y\in(X\setminus A)\alpha$ then $y\leq z$, a contradiction. 
Hence $y\in A\alpha$ and so $B\beta\subseteq A\alpha$, as required. 
\end{proof}

\begin{lemma}\label{lema2L}
Let $\alpha,\beta\in\OP(X)$ be such that $\im(\alpha)=\im(\beta)$.
If $A$ and $B$ are ideals of $\alpha$ and $\beta$, respectively,  
and $A\alpha\subsetneqq B\beta$ then $(X\setminus B)\beta\subseteq (X\setminus A)\alpha$.
\end{lemma}
\begin{proof}
Let $z\in B\beta\setminus A\alpha$. 
Then, as $z\in B\beta\subseteq\im(\beta)=\im(\alpha)=A\alpha\cup(X\setminus A)\alpha$, we can
conclude that $z\in(X\setminus A)\alpha$. 
Let $y\in(X\setminus B)\beta$. Then $y\leq z$. If $y\in A\alpha$, as
$z\in(X\setminus A)\alpha$, then $z\leq y$ and so 
$y=z\not\in A\alpha$, a contradiction. Hence $y\in(X\setminus A)\alpha$, 
since $(X\setminus B)\beta\subseteq\im(\beta)=\im(\alpha)=A\alpha\cup(X\setminus A)\alpha$. 
Thus  $(X\setminus B)\beta\subseteq (X\setminus A)\alpha$, as required.
\end{proof}

%

Now, we prove the description of Green's relation $\GRL$ on $\OP(X)$ announced above. 

\begin{theorem}\label{prop1}
Let $\alpha,\beta\in\OP(X)$. Then $\alpha\GRL\beta$ in $\OP(X)$ if and only if $\im(\alpha)=\im(\beta)$.
\end{theorem}
\begin{proof} 
It remains to prove that $\im(\alpha)=\im(\beta)$ implies $\alpha\GRL\beta$. 
Therefore, suppose that $\im(\alpha)=\im(\beta)$ and 
let $A$ and $B$ be ideals of $\alpha$ and $\beta$, respectively. 

For each $y\in\im(\beta)=\im(\alpha)$, 
choose $z_y\in y\beta^{-1}$ 
such that if $y\in B\beta$ then $z_y\in B$. 
Define a transformation $\gamma\in\T(X)$ by $x\gamma=z_{x\alpha}$, for all $x\in X$. 
Then, $x\gamma\beta=z_{x\alpha}\beta=x\alpha$, for all $x\in X$, i.e. $\alpha=\gamma\beta$. 
Now, our aim is to prove that $\gamma\in\OP(X)$. With this in mind and in view of Lemma \ref{lema1L}, we consider the two possible cases. 

\smallskip

\noindent\textsc{case 1.} $A\alpha\subseteq B\beta$.

Let $C=X\setminus\{x\in X\setminus A\mid x\alpha\in B\beta\}$. 

We begin by proving that $C$ satisfies (OP1) for $\gamma$.  

Let $x_1,x_2\in C$ be such that $x_1\leq x_2$. 
Then either $x_i\in A$ or $x_i\in X\setminus A$ and $x_i\alpha\not\in B\beta$, for $i=1,2$. 

First, suppose that $x_1,x_2\in A$. Then
$x_1\alpha,x_2\alpha\in A\alpha\subseteq B\beta$, 
from which it follows that $z_{x_1\alpha},z_{x_2\alpha}\in B$. 
Moreover, $x_1\alpha\leq x_2\alpha$.  
If $z_{x_2\alpha}<z_{x_1\alpha}$ then 
$x_2\alpha=z_{x_2\alpha}\beta\le z_{x_1\alpha}\beta=x_1\alpha$, 
whence $x_1\alpha=x_2\alpha$ and so $z_{x_1\alpha}=z_{x_2\alpha}$, 
which is a contradiction. 
Hence $x_1\gamma=z_{x_1\alpha}\leq z_{x_2\alpha}=x_2\gamma$.

Secondly, suppose that $x_1,x_2\in X\setminus A$. 
Then $x_1\alpha\leq x_2\alpha$. Additionally,  $x_1\alpha, x_2\alpha\not\in B\beta$ 
from which follows $z_{x_1\alpha},z_{x_2\alpha}\in X\setminus B$. 
As above, if $z_{x_2\alpha}<z_{x_1\alpha}$,  
we may derive a contradiction. 
Hence $x_1\gamma=z_{x_1\alpha}\leq z_{x_2\alpha}=x_2\gamma$.

Finally, suppose that $x_1\in A$ and $x_2\in X\setminus A$. 
Then, as above, we may deduce that $z_{x_1\alpha}\in B$ and 
$z_{x_2\alpha}\in X\setminus B$. Hence 
$x_1\gamma=z_{x_1\alpha}<z_{x_2\alpha}=x_2\gamma$. 

Notice that, since $x_1\le x_2$, we cannot have $x_1\in X\setminus A$
and $x_2\in A$. 

Now, let $x_1,x_2\in X\setminus C$ be such that $x_1\leq x_2$. 
Then $x_1,x_2\in X\setminus A$ and $x_1\alpha,x_2\alpha\in B\beta$. 
Hence $x_1\alpha\leq x_2\alpha$ and $z_{x_1\alpha},z_{x_2\alpha}\in B$. 
Once again, if $z_{x_2\alpha}<z_{x_1\alpha}$ then 
a contradiction can be derived. 
Hence $x_1\gamma=z_{x_1\alpha}\leq z_{x_2\alpha}=x_2\gamma$.

Thus, we proved that $C$ satisfies (OP1) for $\gamma$.  

Next, we prove that $C$ satisfies (OP2) for $\gamma$.  

Let $x_1\in C$ and $x_2\in X\setminus C$. 
Then $x_2\in X\setminus A$ and $x_2\alpha\in B\beta$. 
Regarding $x_1$, we have two cases. 

First, suppose that $x_1\in A$. 
Then, we obtain  $x_1\le x_2$ and $x_1\alpha\geq x_2\alpha$. 
On the other hand, we also have $x_1\alpha\in A\alpha\subseteq B\beta$. 
Thus $z_{x_1\alpha},z_{x_2\alpha}\in B$. 
If $z_{x_1\alpha}<z_{x_2\alpha}$ then again 
a contradiction can be derived. 
Hence $x_1\gamma=z_{x_1\alpha}\geq z_{x_2\alpha}=x_2\gamma$.

Lastly, suppose that $x_1\in X\setminus A$. 
Then $x_1\alpha\not\in B\beta$ and so $x_1\alpha\in (X\setminus B)\beta$ and $z_{x_1\alpha}\in X\setminus B$. 

On the other hand, as $x_2\alpha\in B\beta$, then $z_{x_2\alpha}\in B$. Hence 
$x_1\alpha\le x_2\alpha$ and $x_2\gamma=z_{x_2\alpha}< z_{x_1\alpha}=x_1\gamma$. 

If $x_2\le x_1$ then $x_2\alpha\le x_1\alpha$, since $x_1,x_2\in X\setminus A$, and so $x_1\alpha=x_2\alpha$, 
from which follows that $z_{x_1\alpha}= z_{x_2\alpha}$, a contradiction. Thus, we also have $x_1< x_2$. 
This finishes the proof that $C$ satisfies (OP2) for $\gamma$.  

Therefore, $C$ is an ideal of $\gamma$ (notice that, as $A\subseteq C$, we have $C\ne\emptyset$) and so $\gamma\in\OP(X)$.

\smallskip

\noindent\textsc{case 2.} $B\beta\subsetneqq A\alpha$. 

Notice that, by Lemma \ref{lema2L}, we have $(X\setminus A)\alpha\subseteq (X\setminus B)\beta$.

Let $C=\{x\in A\mid x\alpha\in A\alpha\setminus B\beta\}$. 

Observe that $C\ne\emptyset$, otherwise we would have $A\alpha\subseteq B\beta$, a contradiction. 

We start by proving that $C$ satisfies (OP1) for $\gamma$. 

Let $x_1,x_2\in C$ be such that $x_1\leq x_2$. 
Then $x_1,x_2\in A$ and so $x_1\leq x_2$ implies $x_1\alpha\leq x_2\alpha$. 
Moreover,  $x_1\alpha,x_2\alpha\in A\alpha\setminus B\beta$, from which follows 
$z_{x_1\alpha},z_{x_2\alpha}\in X\setminus B$.
If $z_{x_2\alpha}<z_{x_1\alpha}$ then, once more,  
a contradiction can be derived. 
Thus $x_1\gamma=z_{x_1\alpha}\le z_{x_2\alpha}=x_2\gamma$.

Next, let $x_1,x_2\in X\setminus C$ be such that $x_1\leq x_2$. 
Then, either $x_i\in X\setminus A$ or $x_i\in A$ and $x_i\alpha\in B\beta$, 
for $i=1,2$.

First, suppose that $x_1,x_2\in A$. 
Then $x_1\alpha\leq x_2\alpha$. Additionally, 
$x_1\alpha,x_2\alpha\in B\beta$ and so $z_{x_1\alpha},z_{x_2\alpha}\in B$. 
Once again, 
if $z_{x_2\alpha}<z_{x_1\alpha}$ then 
a contradiction can be derived. 
Thus $x_1\gamma=z_{x_1\alpha}\le z_{x_2\alpha}=x_2\gamma$.

Secondly, suppose that $x_1\in A$ and $x_2\in X\setminus A$. 
Then, as above,  $x_1\alpha\in B\beta$ and so $z_{x_1\alpha}\in B$.  
Moreover, $x_2\alpha\in(X\setminus A)\alpha\subseteq (X\setminus B)\beta$. 
Let us suppose that we also have $x_2\alpha\in B\beta$. 
Then, by Proposition \ref{glued}, we get $x_2\alpha=\max((X\setminus B)\beta)=\min(B\beta)$. 
Since $B\beta\subsetneqq A\alpha$, there exists $y\in A\alpha$ such that $y\not\in B\beta$. 
Hence $y\in(X\setminus B)\beta$ and so $y\leq x_2\alpha$. 
On the other hand, 
as $x_2\alpha\in(X\setminus A)\alpha$ and $y\in A\alpha$, we obtain $x_2\alpha\leq y$. 
Then $y=x_2\alpha\in B\beta$, which is a contradiction. Thus $x_2\alpha\not\in B\beta$ and
so $z_{x_2\alpha}\in X\setminus B$. Therefore, 
$x_1\gamma=z_{x_1\alpha}<z_{x_2\alpha}=x_2\gamma$.

Again, notice that, since $x_1\le x_2$, we cannot have $x_1\in X\setminus A$
and $x_2\in A$. Therefore, finally, we suppose that $x_1,x_2\in X\setminus A$.
Then, we have $x_1\alpha, x_2\alpha\in(X\setminus A)\alpha\subseteq (X\setminus B)\beta$ 
and, as above, we may deduce that $x_1\alpha, x_2\alpha\not\in B\beta$, from which follows that 
$z_{x_1\alpha},z_{x_2\alpha}\in X\setminus B$. 
Moreover, we have $x_1\alpha\le x_2\alpha$. 
So, over again, 
if $z_{x_2\alpha}<z_{x_1\alpha}$ then 
a contradiction can be derived. 
Thus $x_1\gamma=z_{x_1\alpha}\le z_{x_2\alpha}=x_2\gamma$.

Thus, we proved that $C$ satisfies (OP1) for $\gamma$.  

It remains to prove that $C$ satisfies (OP2) for $\gamma$.  

Let $x_1\in C$ and $x_2\in X\setminus C$. 
Then $x_1\in A$ and $x_1\alpha\not\in B\beta$ and so $z_{x_1\alpha}\in X\setminus B$.   
On the other hand, either $x_2\in X\setminus A$ or $x_2\in A$ and $x_2\alpha\in B\beta$.  

First, suppose that $x_2\in X\setminus A$. As $x_1\in A$, we have $x_1\le x_2$ and $x_1\alpha\ge x_2\alpha$. 
On the other hand,  $x_2\alpha\in(X\setminus A)\alpha\subseteq (X\setminus B)\beta$ 
and, once again, we may deduce that $x_2\alpha\not\in B\beta$, from which follows that 
$z_{x_2\alpha}\in X\setminus B$. 
If $z_{x_1\alpha}<z_{x_2\alpha}$ then $x_1\alpha=z_{x_1\alpha}\beta\le z_{x_2\alpha}\beta=x_2\alpha$ 
(since $z_{x_1\alpha}, z_{x_2\alpha}\in X\setminus B$), 
whence $x_1\alpha=x_2\alpha$ and so $z_{x_1\alpha}=z_{x_2\alpha}$, which is a contradiction. 
Thus $x_1\gamma=z_{x_1\alpha}\ge z_{x_2\alpha}=x_2\gamma$.

Finally, we suppose that $x_2\in A$ and so  $x_2\alpha\in B\beta$. 
As $x_1\alpha\not\in B\beta$, we must have $x_1\alpha\le x_2\alpha$. 
If $x_1>x_2$ then $x_1\alpha\ge x_2\alpha$ (since $x_1,x_2\in A$) 
and so $x_1\alpha=x_2\alpha\in B\beta$, which is a contradiction. 
Hence $x_1\le x_2$. 
On the other hand, from $x_1\alpha\not\in B\beta$ and $x_2\alpha\in B\beta$, 
it follows that $z_{x_1\alpha}\in X\setminus B$ and $z_{x_2\alpha}\in B$ and so 
$x_1\gamma=z_{x_1\alpha}> z_{x_2\alpha}=x_2\gamma$.

Therefore, $C$ is an ideal of $\gamma$ and so, also in this case, we have $\gamma\in\OP(X)$.

\smallskip 

So, we have showed that $\alpha=\gamma\beta$, with $\gamma\in\OP(X)$. 
Analogously, we can prove that $\beta=\lambda\alpha$, for some $\lambda\in\OP(X)$. 
Therefore $\alpha\GRL\beta$ as required. 
\end{proof} 

\medskip

In order to describe the Green's relation $\GRR$ in $\OP(X)$, 
we need to introduce the following concepts. 

\begin{definition}
Let $A$ and $B$ be two subsets of $X$ and let $\theta:A\longrightarrow B$ be a mapping.  

We say that $\theta$ is \textit{completable} in $\OP(X)$ if there exists $\gamma\in\OP(X)$ such that 
$x\gamma=x\theta$, for all $x\in A$. 
To such a transformation $\gamma$ (not necessarily unique) we call a \textit{complete extension} 
of $\theta$ in $\OP(X)$. 

If $\theta:A\longrightarrow B$ is a bijection, we say that $\theta$ is \textit{bicompletable} in $\OP(X)$ 
if both $\theta$ and its inverse $\theta^{-1}:B\longrightarrow A$ are completable in $\OP(X)$. 
\end{definition}

By the \textit{inverse} $\theta^{-1}$ of an injective mapping $\theta:A\longrightarrow B$,  
we mean the inverse mapping $\theta^{-1}:\im(\theta)\longrightarrow A$ 
of the bijection $\theta:A\longrightarrow \im(\theta)$. 
Thus, we say that an injective mapping $\theta:A\longrightarrow B$ admits a completable inverse in $\OP(X)$ 
if its inverse $\theta^{-1}:\im(\theta)\longrightarrow A$ is completable in $\OP(X)$. 

\medskip 

Let $\alpha,\beta\in\PT(X)$ be such that $\ker\alpha=\ker\beta$. 
Let $\phi\subseteq \im(\alpha)\times\im(\beta)$ be the relation defined by
$$
(a,b)\in\phi\,\text{ if and only if }\,a\alpha^{-1}=b\beta^{-1},\,\text{ for all
}\,(a,b)\in\im(\alpha)\times\im(\beta).
$$ 
Then, it is a routine matter to show that 
$\phi:\im(\alpha)\longrightarrow \im(\beta)$ is a bijection such that 
$\alpha=\beta\phi^{-1}$ and $\beta=\alpha\phi$. 
Under these conditions, we say that $\phi$ is the \textit{canonical bijection} from $\im(\alpha)$ into $\im(\beta)$. 
Notice that, given $a\in\im(\alpha)$, we have $a\phi=x\beta$, for any $x\in a\alpha^{-1}$. 

\smallskip 

Now, we can present our description of the relation $\GRR$ in $\OP(X)$. 

\begin{theorem}\label{prop3}
Let $\alpha,\beta\in\OP(X)$. Then
$\alpha\GRR\beta$ in $\OP(X)$ if and only if
$\ker(\alpha)=\ker(\beta)$ and the canonical bijection
$\phi:\im(\alpha)\longrightarrow \im(\beta)$ is bicompletable in
$\OP(X)$.
\end{theorem}
\begin{proof}
First, suppose that $\alpha\GRR\beta$ in $\OP(X)$ and 
let $\gamma,\lambda\in\OP(X)$ be such that $\alpha=\beta\gamma$ and
$\beta=\alpha\lambda$. As $\alpha\GRR\beta$ in $\OP(X)$, we also have 
$\alpha\GRR\beta$ in $\T(X)$ and so $\ker(\alpha)=\ker(\beta)$.
Let $\phi:\im(\alpha)\longrightarrow \im(\beta)$ be the canonical
bijection. 
Let $a\in\im(\alpha)$. Then $a\alpha^{-1}=(a\phi)\beta^{-1}$ and, 
by taking $x\in a\alpha^{-1}$, we have $a\lambda=x\alpha\lambda=x\beta=a\phi$. 
Hence $\lambda$ is a complete extension of $\phi$. 
Similarly, we may show that $\gamma$ is a complete extension in $\OP(X)$ of $\phi^{-1}$.
Therefore $\phi$ is bicompletable  in $\OP(X)$.

Conversely, suppose that $\ker(\alpha)=\ker(\beta)$ and the
canonical bijection $\phi:\im(\alpha)\longrightarrow \im(\beta)$
is bicompletable in $\OP(X)$. 
Let $\lambda$ and $\gamma$ 
be complete extensions in $\OP(X)$ of $\phi$ and $\phi^{-1}$, respectively. 
Since $\beta=\alpha\phi$, $\alpha=\beta\phi^{-1}$, $\dom(\phi)=\im(\alpha)$ and 
$\dom(\phi^{-1})=\im(\beta)$, we have $\beta=\alpha\lambda$ and $\alpha=\beta\gamma$, 
whence $\alpha\GRR\beta$, as required.
\end{proof}

From Theorem \ref{prop1} and Theorem \ref{prop3} it follows immediately: 

\begin{corollary}\label{Hrelation}
Let $\alpha,\beta\in\OP(X)$. Then
$\alpha\GRH\beta$ in $\OP(X)$ if and only if
$\im(\alpha)=\im(\beta)$, $\ker(\alpha)=\ker(\beta)$ and the
canonical bijection $\phi:\im(\alpha)\longrightarrow \im(\beta)$
is bicompletable in $\OP(X)$.
\end{corollary}

Regarding the relation $\GRD$, we have: 

\begin{theorem}\label{prop4}
Let $\alpha,\beta\in\OP(X)$. Then
$\alpha\GRD\beta$ in $\OP(X)$ if and only if there exists a
bijective mapping  $\theta:\im(\alpha)\longrightarrow \im(\beta)$ which is bicompletable in $\OP(X)$. 
\end{theorem} 
\begin{proof}
Suppose that $\alpha\GRD\beta$. Then, there exists
$\gamma\in\OP(X)$ such that $\alpha\GRR\gamma$ and
$\gamma\GRL\beta$. By Theorem \ref{prop3}, we have that 
$\ker(\alpha)=\ker(\gamma)$ and the canonical bijection
$\phi:\im(\alpha)\longrightarrow \im(\gamma)$ is bicompletable
in $\OP(X)$. On the other hand, by Theorem \ref{prop1}, we
have $\im(\gamma)=\im(\beta)$. Thus, we obtain a bijection
$\phi:\im(\alpha)\longrightarrow \im(\gamma)=\im(\beta)$ which
is bicompletable in $\OP(X)$. 

Conversely, suppose that there exists a bijection
$\theta:\im(\alpha)\longrightarrow\im(\beta)$ which is bicompletable in
$\OP(X)$. Let $\xi$ be a complete extension of $\theta$ in $\OP(X)$. 
Let $\gamma=\alpha\xi$. Then $\gamma\in\OP(X)$. 
Moreover, $\gamma=\alpha\theta$ and $\im(\gamma)=(\im(\alpha))\theta=\im(\beta)$. 
Hence, by Theorem \ref{prop1}, we have 
$\gamma\GRL\beta$.
On the other hand, let $x,y\in X$. Since $\theta$ is a bijection, we have
$$
x\alpha=y\alpha\,\Leftrightarrow\,(x\alpha)\theta=(y\alpha)\theta\,\Leftrightarrow\,x\gamma=y\gamma.
$$
Then $\ker(\alpha)=\ker(\gamma)$.  
Let us consider the canonical bijection $\phi:\im(\alpha)\longrightarrow\im(\gamma)$.  
Next, we prove that $\phi$ is bicompletable in $\OP(X)$ by showing that $\phi=\theta$. 
Let $a\in\im(\alpha)$ and take $x\in a\alpha^{-1}$. Then, 
$a\alpha^{-1}=(a\phi)\gamma^{-1}$ and 
$a\phi=x\gamma=x\alpha\theta=a\theta$. 
Thus, $\phi=\theta$ and so $\phi$ is bicompletable in $\OP(X)$, 
whence $\alpha\GRR\gamma$, by Theorem \ref{prop3}. 

Therefore $\alpha\GRD\beta$, as required.
\end{proof}

To finish the study of the Green's relations in $\OP(X)$, we give the following description of $\GRJ$. 

\begin{theorem}\label{prop5}
Let $\alpha,\beta\in\OP(X)$. Then $\alpha\GRJ\beta$ in $\OP(X)$ if and only if there exist
injective mappings $\theta:\im(\alpha)\longrightarrow \im(\beta)$ and
$\tau:\im(\beta)\longrightarrow \im(\alpha)$ admitting completable
inverses in $\OP(X)$.
\end{theorem}

\begin{proof}
Suppose that $\alpha\GRJ\beta$. Then, there exist  
$\lambda,\gamma\in\OP(X)$ such that
$\alpha=\lambda\beta\gamma$.

For each $a\in\im(\alpha)=(\im(\lambda\beta))\gamma$, 
choose $w_a\in\im(\lambda\beta)\cap a\gamma^{-1}$. 
Then, define a mapping $\theta:\im(\alpha)\longrightarrow X$ by $a\theta=w_a$,
for all $a\in\im(\alpha)$. Notice that $\im(\theta)\subseteq\im(\beta)$. 

Next, we prove that $\theta$ is injective. 
Let $a,b\in\im(\alpha)$ be such that $a\theta=b\theta$. 
Then $w_a=w_b$ and, as $w_a\in a\gamma^{-1}$ and $w_b\in b\gamma^{-1}$, we obtain 
$a=w_a\gamma=w_b\gamma=b$. Hence, $\theta$ is an injective mapping. 

Finally, we show that $\theta:\im(\alpha)\longrightarrow \im(\beta)$ admits a completable inverse in 
$\OP(X)$. Let $w\in\dom(\theta^{-1})=\im(\theta)$ and take $a=w\theta^{-1}\in\im(\alpha)$. 
Then $w=a\theta=w_a$ and so $w\theta^{-1}=a=w_a\gamma=w\gamma$. 
Thus $\gamma$ is a complete extension of $ \theta^{-1}$ in $\OP(X)$.

Similarly, 
by taking $\delta,\xi\in\OP(X)$ such that $\beta=\delta\alpha\xi$, 
we can construct an injective mapping 
$\tau:\im(\beta)\longrightarrow\im(\alpha)$
that admits a completable inverse in $\OP(X)$.

\smallskip 

Conversely, let $\theta:\im(\alpha)\longrightarrow\im(\beta)$ be an injective
mapping such that its inverse $\theta^{-1}:\im(\theta)\longrightarrow\im(\alpha)$ is completable in
$\OP(X)$. Let $\gamma\in\OP(X)$ be a complete extension of $\theta^{-1}$.  
Notice that, by Proposition \ref{rPOP}, we have $\theta^{-1}\in\POP(X)$ and so, by Proposition \ref{iPOP}, it follows that $\theta\in\POP(X)$. 

Let $B$ be an ideal of $\beta$ and consider $C=(X\setminus B)\beta$. 
For each $c\in C$, choose $z_c\in c\beta^{-1}\cap (X\setminus B)$. 
For each $c\in\im(\beta)\setminus C$, choose $z_c\in c\beta^{-1}\cap B$ 
(notice that $\im(\beta)\setminus C\subseteq B\beta$).
Then, define a mapping $\bar\beta\in\PT(X)$, with $\dom(\bar\beta)=\im(\beta)$, 
by $c\bar\beta=z_c$, for all $c\in\im(\beta)$.  
Hence, it is a routine matter to show that, if $C\ne\emptyset$ then $\bar\beta$ admits $C$ as an ideal, 
and, if $C=\emptyset$ then $\bar\beta$ is order-preserving. 
Thus, in both cases, we have $\bar\beta\in\POP(X)$.  

Now, let $\lambda=\alpha\theta\bar\beta$. Then, as $\alpha,\theta,\bar\beta\in\POP(X)$, 
we have $\lambda\in\POP(X)$. Additionally, $\dom(\lambda)=X$ and so $\lambda\in\OP(X)$. 

Next, take $x\in X$. 
Then 
$$
x\lambda\beta\gamma=
x\alpha\theta\bar\beta\beta\gamma= 
z_{x\alpha\theta}\beta\gamma=x\alpha\theta\gamma=x\alpha\theta\theta^{-1}=x\alpha. 
$$
Thus $\alpha=\lambda\beta\gamma$.

Similarly, by considering an injective mapping $\tau:\im(\beta)\longrightarrow \im(\alpha)$ 
that admits a completable inverse in $\OP(X)$, we can construct transformations $\delta,\xi\in\OP(X)$ such that
$\beta=\delta\alpha\xi$. Therefore $\alpha\GRJ\beta$, as required.
\end{proof}

\medskip 

Now, we recall the following descriptions of the Green's relations in the semigroups $\PT(X)$ and $\I(X)$ (see \cite{Howie:1995}). 
Let $S\in\{\PT(X),\I(X)\}$ and let $\alpha,\beta\in S$. Then, in $S$, we have: 
\begin{itemize}
\item $\alpha\GRL\beta$ if and only if $\im(\alpha)=\im(\beta)$;
\item $\alpha\GRR\beta$ if and only if $\ker(\alpha)=\ker(\beta)$;
\item $\alpha\GRD\beta$ if and only if $|\im(\alpha)|=|\im(\beta)|$;
\item $\GRJ=\GRD$. 
\end{itemize}

Notice that, if $\alpha,\beta\in\I(X)$ then $\ker(\alpha)=\ker(\beta)$ if and only if $\dom(\alpha)=\dom(\beta)$. 

\smallskip 

The situation regarding the Green's relations of the partial, and the partial injective, counterparts of $\OO(X)$ and $\OP(X)$, 
such as for the regularity, is much more simple than for these semigroups.  

In fact, beginning with the Green's relations $\GRR$ and $\GRL$, 
as an immediate consequence of the regularity of each of the semigroups (recall Theorem \ref{regPOP}), we have:  

\begin{proposition}\label{rellr}
Let $S\in\{\PO(X),\POP(X),\POI(X),\POPI(X)\}$. 
Let  $\alpha,\beta\in S$. Then, in $S$, we have:
\begin{enumerate}
\item $\alpha\GRL\beta$ if and only if $\im(\alpha)=\im(\beta)$;
\item $\alpha\GRR\beta$ if and only if $\ker(\alpha)=\ker(\beta)$; 
\item $\alpha\GRH\beta$ if and only if $\im(\alpha)=\im(\beta)$ and $\ker(\alpha)=\ker(\beta)$. 
\end{enumerate}
\end{proposition} 

\smallskip 

Observe that, if $S\in\{\POI(X),\POPI(X)\}$ then, in $S$, we also have:  
\begin{enumerate}
\item $\alpha\GRR\beta$ if and only if $\dom(\alpha)=\dom(\beta)$;
\item $\alpha\GRH\beta$ if and only if $\dom(\alpha)=\dom(\beta)$ and $\im(\alpha)=\im(\beta)$. 
\end{enumerate}

\medskip 

In order to characterize the Green's relation $\GRD$, we first prove the following lemma. 

\begin{lemma}\label{cabij}
Let $\alpha,\beta\in\POP(X)$ be such that $\ker(\alpha)=\ker(\beta)$ and let $\phi:\im(\alpha)\longrightarrow \im(\beta)$ be the canonical bijection from $\im(\alpha)$ into $\im(\beta)$. 
Then $\phi$ is an orientation-preserving bijection. Moreover, if $\alpha,\beta\in\PO(X)$ then $\phi$ is an order-isomorphism.
\end{lemma} 
\begin{proof}
First, observe that $\dom(\alpha)=\dom(\beta)$, since $\ker(\alpha)=\ker(\beta)$. 
Let $Y$ be an ideal of $\alpha$ and $Z$ be an ideal of $\beta$. As $Y$ and $Z$ are order ideals of $\dom(\alpha)$, then $Y\subseteq Z$ or $Z\subseteq Y$. Without loss of generality, we suppose that $Y\subseteq Z$. Let $W=(Z\setminus Y)\alpha$. 

\smallskip 

We begin by showing that $\phi$ is order-preserving on $W$. 
Let $a_1,a_2\in W$ be such that $a_1<a_2$. 
Then, there exist $x_1,x_2\in Z\setminus Y$ such that $x_1\in a_1\alpha^{-1}$ and $x_2\in a_2\alpha^{-1}$. 
As $Z\setminus Y\subseteq \dom(\alpha)\setminus Y$, if $x_2<x_1$ then $a_2=x_2\alpha\le x_1\alpha=a_1$, which is a contradiction. 
Hence, $x_1\le x_2$ and, as $Z\setminus Y\subseteq Z$, we have $a_1\phi=x_1\beta\le x_2\beta=a_2\phi$. 

\smallskip 

Next, we prove that $\phi$ is order-preserving on $\im(\alpha)\setminus W$. 
Let $a_1,a_2\in \im(\alpha)\setminus W$ be such that $a_1<a_2$. 
Take $x_1,x_2\in \dom(\alpha)$ such that $x_1\in a_1\alpha^{-1}$ and $x_2\in a_2\alpha^{-1}$. Clearly, $x_1,x_2\not\in Z\setminus Y$ and so 
$x_1,x_2\in Y\cup (\dom(\alpha)\setminus Z)$. 
Suppose that $x_1,x_2\in Y$. If $x_2<x_1$ then $a_2=x_2\alpha\le x_1\alpha=a_1$, which is a contradiction. 
Hence, $x_1\le x_2$ and, as $Y\subseteq Z$, we have $a_1\phi=x_1\beta\le x_2\beta=a_2\phi$. 
Secondly, consider that $x_1\in Y$ and $x_2\in\dom(\alpha)\setminus Z$. 
As $\dom(\alpha)\setminus Z\subseteq\dom(\alpha)\setminus Y$, we have $x_1\in Y$ and $x_2\in\dom(\alpha)\setminus Y$, whence $a_1=x_1\alpha\ge x_2\alpha=a_2$, which is a contradiction. So, this case does not occur. 
In third place,  if $x_1\in\dom(\alpha)\setminus Z$ and $x_2\in Y$ then, as $Y\subseteq Z$, 
$x_1\in\dom(\beta)\setminus Z$ and $x_2\in Z$ and so $a_1\phi=x_1\beta\le x_2\beta=a_2\phi$. 
Finally, suppose that $x_1,x_2\in\dom(\alpha)\setminus Z$. As $\dom(\alpha)\setminus Z\subseteq\dom(\alpha)\setminus Y$, 
if $x_2<x_1$ then  $a_2=x_2\alpha\le x_1\alpha=a_1$, which is a contradiction. 
Hence, $x_1,x_2\in\dom(\beta)\setminus Z$ and $x_1\le x_2$ , from which follows that $a_1\phi=x_1\beta\le x_2\beta=a_2\phi$. 

\smallskip 

Now, let $a_1\in W$ and $a_2\in \im(\alpha)\setminus W$. We aim to show that $a_1\le a_2$ and $a_1\phi\ge a_2\phi$. 
Take $x_1\in Z\setminus Y$ and $x_2\in Y\cup (\dom(\alpha)\setminus Z)$ such that $x_1\in a_1\alpha^{-1}$ and $x_2\in a_2\alpha^{-1}$. 
First, suppose that $x_2\in Y$. As $x_1\in Z\setminus Y\subseteq \dom(\alpha)\setminus Y$, we have $x_2<x_1$ and 
$a_2=x_2\alpha\ge x_1\alpha=a_1$. 
On the other hand, $x_1\in Z\setminus Y\subseteq Z$, $x_2\in Y\subseteq Z$ and $x_2< x_1$ imply $a_2\phi=x_2\beta\le x_1\beta=a_1\phi$. 
Secondly, suppose that $x_2\in \dom(\alpha)\setminus Z$. 
Then $x_1\in Z\setminus Y\subseteq Z$ and $x_2\in \dom(\beta)\setminus Z$, whence $x_1<x_2$ and $a_1\phi=x_1\beta\ge x_2\beta=a_2\phi$. 
Furthermore, $x_1\in Z\setminus Y\subseteq \dom(\alpha)\setminus Y$, $x_2\in \dom(\alpha)\setminus Z\subseteq\dom(\alpha)\setminus Y$ and 
$x_1<x_2$ imply $a_1=x_1\alpha\le x_2\alpha=a_2$. 

\smallskip

Thus, if $W\ne\emptyset$, we proved that $W$ is an ideal of $\phi$. On the other hand, if $W=\emptyset$ then we proved that $\phi$ is order-preserving (on $\im(\alpha)$). In both cases,  we have $\phi\in\POPI(X)$. Moreover, if $\alpha,\beta\in\PO(X)$ then $W=\emptyset$ and so 
$\phi\in\POI(X)$, as required. 
\end{proof} 

\begin{proposition}\label{reld}
Let $S\in\{\PO(X),\POI(X)\}$ [respectively, $S\in\{\POP(X),\POPI(X)\}$]. 
Let  $\alpha,\beta\in S$. Then $\alpha\GRD\beta$ in $S$ if and only if there exists an order-isomorphism 
[respectively, an orientation-preserving bijection] $\theta:\im(\alpha)\longrightarrow \im(\beta)$. 
\end{proposition} 
\begin{proof}
Suppose that $\alpha\GRD\beta$. Then, there exists
$\gamma\in S$ such that $\alpha\GRR\gamma$ and
$\gamma\GRL\beta$. By Proposition \ref{rellr}, we have that 
$\ker(\alpha)=\ker(\gamma)$ and $\im(\gamma)=\im(\beta)$. 
Then, by Lemma \ref{cabij}, the canonical bijection $\phi:\im(\alpha)\longrightarrow \im(\gamma)=\im(\beta)$ from $\im(\alpha)$ into $\im(\gamma)$ 
is an order-isomorphism [respectively, an orientation-preserving bijection]. 

Conversely, suppose that there exists an order-isomorphism 
[respectively, an orientation-preserving bijection] $\theta:\im(\alpha)\longrightarrow \im(\beta)$. 
Let $\gamma=\alpha\theta$. Then $\gamma\in S$. 
Moreover, $\im(\gamma)=(\im(\alpha))\theta=\im(\beta)$. 
On the other hand, $\dom(\gamma)=\dom(\theta)\alpha^{-1}=\im(\alpha)\alpha^{-1}=\dom(\alpha)$ and, 
since $\theta$ is a bijection, we have
$$
x\alpha=y\alpha\,\Leftrightarrow\,(x\alpha)\theta=(y\alpha)\theta\,\Leftrightarrow\,x\gamma=y\gamma, 
$$
for all $x,y\in\dom(\alpha)=\dom(\gamma)$. 
Then $\ker(\alpha)=\ker(\gamma)$.  
Hence, by Proposition \ref{rellr}, we have 
$\gamma\GRL\beta$ and $\alpha\GRR\gamma$
and so $\alpha\GRD\beta$, as required.
\end{proof}

Finally, for $\GRJ$, we have: 

\begin{proposition}\label{relj}
Let $S\in\{\PO(X),\POI(X)\}$ [respectively, $S\in\{\POP(X),\POPI(X)\}$]. 
Let  $\alpha,\beta\in S$. Then $\alpha\GRJ\beta$ in $S$ if and only if there exist order-preserving 
[respectively, orientation-preserving] injections $\theta:\im(\alpha)\longrightarrow \im(\beta)$ and
$\tau:\im(\beta)\longrightarrow \im(\alpha)$.
\end{proposition}
\begin{proof}
First, suppose that $\alpha\GRJ\beta$. 
Let $\lambda,\gamma\in S$ be such that $\alpha=\lambda\beta\gamma$.

For each $a\in\im(\alpha)=(\im(\lambda\beta))\gamma$, 
choose $w_a\in\im(\lambda\beta)\cap a\gamma^{-1}$. 
Then, define a mapping $\theta:\im(\alpha)\longrightarrow X$ by $a\theta=w_a$,
for all $a\in\im(\alpha)$. Notice that $\im(\theta)\subseteq\im(\beta)$. 

Next, we prove that $\theta$ is injective. 
Let $a,b\in\im(\alpha)$ be such that $a\theta=b\theta$. 
Then $w_a=w_b$ and, as $w_a\in a\gamma^{-1}$ and $w_b\in b\gamma^{-1}$, we obtain 
$a=w_a\gamma=w_b\gamma=b$. Hence, $\theta$ is an injective mapping. 

Let us consider the inverse mapping $\theta^{-1}:\im(\theta)\longrightarrow\im(\alpha)$ of $\theta$. 
Let $w\in\dom(\theta^{-1})=\im(\theta)$ and take $a=w\theta^{-1}\in\im(\alpha)$. 
Then $w=a\theta=w_a$ and so $w\theta^{-1}=a=w_a\gamma=w\gamma$.  
Therefore, $\theta^{-1}$ is a restriction of $\gamma$ and so $\theta^{-1}$ is an order-preserving transformation  
[respectively, orientation-preserving transformation, by Proposition \ref{rPOP}]. 
Thus, as $\theta:\im(\alpha)\longrightarrow\im(\theta)$ is the inverse of $\theta^{-1}$, by Proposition \ref{iPOP}, $\theta$ also is an order-preserving  
[respectively, orientation-preserving] transformation.  

Similarly, 
by taking $\delta,\xi\in\OP(X)$ such that $\beta=\delta\alpha\xi$, 
we can construct an order-preserving 
[respectively, orientation-preserving] injection $\tau:\im(\beta)\longrightarrow \im(\alpha)$. 

\smallskip 

Conversely, let $\theta:\im(\alpha)\longrightarrow\im(\beta)$ be an injective 
order-preserving [respectively, orientation-preserving] transformation.  
Then, by Proposition \ref{iPOP}, $\theta^{-1}:\im(\theta)\longrightarrow\im(\alpha)$ also is an order-preserving  
[respectively, orientation-preserving] transformation.  

Let $B=\dom(\beta)$ [respectively, $B$ be an ideal of $\beta$] and consider $C=(\dom(\beta)\setminus B)\beta$. 
For each $c\in C$, choose $z_c\in c\beta^{-1}\cap (\dom(\beta)\setminus B)$. 
For each $c\in\im(\beta)\setminus C$, choose $z_c\in c\beta^{-1}\cap B$ 
(notice that $\im(\beta)\setminus C\subseteq B\beta$).
Then, define a mapping $\bar\beta\in\PT(X)$, with $\dom(\bar\beta)=\im(\beta)$, 
by $c\bar\beta=z_c$, for all $c\in\im(\beta)$.  
Hence, it is a routine matter to show that $\bar\beta$ is order-preserving 
[respectively, if $C\ne\emptyset$ then $\bar\beta$ admits $C$ as an ideal, 
and, if $C=\emptyset$ then $\bar\beta$ is order-preserving. 
Thus, in both cases, we have that $\bar\beta$ is orientation-preserving].  

Now, let $\lambda=\alpha\theta\bar\beta$. Then, as $\alpha,\theta,\bar\beta\in S$, 
we have $\lambda\in S$.  Take $x\in \dom(\alpha)$. 
Then 
$$
x\lambda\beta\theta^{-1}=
x\alpha\theta\bar\beta\beta\theta^{-1}= 
z_{x\alpha\theta}\beta\theta^{-1}=x\alpha\theta\theta^{-1}=x\alpha. 
$$
Thus $\alpha=\lambda\beta\theta^{-1}$. Notice that $\theta^{-1}\in S$. 

Similarly, by considering an injective order-preserving [respectively, orientation-preserving] transformation $\tau:\im(\beta)\longrightarrow \im(\alpha)$, we can construct transformations $\delta,\xi\in S$ such that
$\beta=\delta\alpha\xi$. Therefore $\alpha\GRJ\beta$, as required.
\end{proof}

\medskip 

We finish this paper by showing that, such as for $\OO(X)$ (see \cite{Fernandes&al:2014}),   
for any 
$$
S\in\{\OP(X), \PO(X), \POP(X), \POI(X), \POPI(X)\},
$$ 
we may have $\GRD\subsetneq\GRJ$ in $S$. 

\smallskip 

In the following lemma and examples, we consider the set of real numbers $\R$ equipped with the usual order. 

\begin{lemma}\label{exes}
Let $a,b,c,d\in\R$ be such that $a<b$ and $c<d$. 
Then, there exists no orientation-preserving bijection from the interval $I=]a,b[$ into the interval $J=[c,d]$. 
\end{lemma}
\begin{proof}
For a contradiction, suppose there exists an orientation-preserving bijection $\theta:I\longrightarrow J$. 
Let $Y$ be an ideal of $\theta$. If $Y=I$ then $\theta$ would be an order-preserving bijection from $I$ into $J$, which is not possible 
(for instance, because $J$ has a minimum and $I$ does not). 
Hence, $Y$ is a proper order ideal of $I$ and $I\setminus Y$ is a non-empty order filter of $I$. 
Since $I$ has no minimum or maximum, then $Y$ has no minimum and $I\setminus Y$ has no maximum. 
Hence, as $\theta$ is order-preserving both in $Y$ and in $I\setminus Y$, 
$Y\theta$ also has no minimum and $(I\setminus Y)\theta$ also has no maximum. 
On the other hand, since each element of $(I\setminus Y)\theta$ is a lower bound of $Y\theta$ and 
each element of $Y\theta$ is an upper bound of $(I\setminus Y)\theta$, 
we have that $(I\setminus Y)\theta$ has a supremum $c'\in I$, $Y\theta$ has an infimum $d'\in I$ and $c'\le d'$. 
Since $(I\setminus Y)\theta$ has no maximum, $c'\not\in (I\setminus Y)\theta$ and so,
as $J=I\theta=(I\setminus Y)\theta\cup Y\theta$, we have $c'\in Y\theta$, whence $d'\le c'$. 
Therefore, $d'=c'$ and so $d'$ is the minimum of $Y\theta$, which is a contradiction.  
Thus, there exists no orientation-preserving bijection $\theta:I\longrightarrow J$, as required. 
\end{proof} 

\smallskip 

\begin{example}\em 
Let  $\alpha,\beta\in\T(\R)$ be defined by $x\alpha=\arctan(x)$, for $x\in\R$, and 
$$
x\beta=\left\{
\begin{array}{cl}
-1 & \mbox{if $x<-1$}\\
x & \mbox{if $-1\le x\le1$}\\
1& \mbox{if $x>1$.}
\end{array}
\right. 
$$
Then $\alpha,\beta\in\OO(\R)$ and so, in particular, $\alpha,\beta\in\OP(\R)$. 
In \cite{Fernandes&al:2014} the authors have showed that  $(\alpha,\beta)\not\in\GRD$ and $(\alpha,\beta)\in\GRJ$ in $\OO(\R)$. 
Hence, as $\OO(\R)$ is a subsemigroup of $\OP(\R)$, we may immediately deduce that also $(\alpha,\beta)\in\GRJ$ in $\OP(\R)$. 
On the other hand, $\im(\alpha)=\left]-\pi/2,\pi/2\right[$ and $\im(\beta)=\left[-1,1\right]$ and so, by the above lemma, 
there exists no orientation-preserving bijection $\theta:\im(\alpha)\longrightarrow\im(\beta)$. Therefore, 
in view of Proposition \ref{rPOP}, there exists no (bi)completable in $\OP(\R)$ bijection $\theta:\im(\alpha)\longrightarrow\im(\beta)$. 
Thus, $(\alpha,\beta)\not\in\GRD$ in $\OP(\R)$ and so we have $\GRD\subsetneq\GRJ$ in $\OP(\R)$. 
\end{example}

\smallskip 

Now, observe that $\POI(X)$ is a subsemigroup of $S$, for any $S\in\{\PO(X),\POP(X),\POPI(X)\}$. 
Furthermore, $S$ is a subsemigroup of $\POP(X)$, for any $S\in\{\PO(X),\POI(X),\POPI(X)\}$. 

\smallskip 

\begin{example}\em 
Let $a,b,c,d\in\R$ be such that $a<b$ and $c<d$. Let $\alpha$ and $\beta$ be the partial identities on the intervals $]a,b[$ and $[c,d]$, respectively.  Then $\alpha,\beta\in\POI(\R)$ and so $\alpha,\beta\in S$, for any $S\in\{\PO(\R),\POP(\R),\POPI(\R)\}$. 
Moreover, $\im(\alpha)=]a,b[$ and $\im(\beta)=[c,d]$. 

By Lemma \ref{exes} there exists no orientation-preserving bijection from $\im(\alpha)$ into $\im(\beta)$ and so, by Proposition \ref{reld}, 
$(\alpha,\beta)\not\in\GRD$ in $\POP(\R)$. 
Therefore, also $(\alpha,\beta)\not\in\GRD$ in $S$,  for any $S\in\{\PO(\R),\POI(\R),\POPI(\R)\}$. 

On the other hand, it is clear that there exist order-preserving injections $\theta:]a,b[\longrightarrow [c,d]$ and
$\tau:[c,d]\longrightarrow ]a,b[$ 
(for instance, $x\theta=\frac{1}{3(b-a)}((d-c)x+2bc-2ad+bd-ac)$, for $x\in ]a,b[$, and 
$x\tau=\frac{1}{3(d-c)}((b-a)x+2ad-2bc+bd-ac)$, for $x\in [c,d]$) and so, 
by Proposition  \ref{relj}, $(\alpha,\beta)\in\GRJ$ in $\POI(\R)$. 
Therefore, also $(\alpha,\beta)\in\GRJ$ in $S$,  for any $S\in\{\PO(\R),\POP(\R),\POPI(\R)\}$. 
Thus $\GRD\subsetneq\GRJ$ in $S$, for any $S\in\{\PO(\R),\POP(\R),\POI(\R),\POPI(\R)\}$. 
\end{example} 

\begin{question}
Do we also have $\GRD\subsetneq\GRJ$ in $S$, for any/some $S\in\{\OP(\Z),\POPI(\Z),\POP(\Z)\}$? 
And, in general, what happens for an arbitrary infinite countable chain? 
\end{question}

\section{Some more open problems}

Recall that, for a subset $A$ of a semigroup $S$, the \textit{relative rank} of $S$ modulo $A$ is the minimum cardinality of a 
subset $B$ of $S$ such that $\langle A \cup B \rangle = S$. This cardinal is denoted by $\rank(S:A)$.

The notion of relative rank was introduced by Ru\v skuc in \cite{Ruskuc:1994}, who proved  that the rank of a finite Rees matrix semigroup
$\mathcal{M}[G; I,\Lambda; P]$, with the sandwich matrix $P$ in normal form, is equal to
$\textrm{max}\{|I|,|\Lambda|, \rank (G : H)\}$, where $H$ is the subgroup of $G$ generated by the
entries of $P$.
In \cite{Howie&Ruskuc&Higgins:1998}, Howie et al. considered the relative ranks of the full transformation
semigroup $\T(X)$ on $X$, where $X$ is an infinite set, modulo some distinguished subsets of $\T(X)$. 
For instance, they showed that $\rank (\T(X) : \S(X)) = 2$ and $\rank (\T(X) : J) = 0$, where $J$ is the top $\mathcal{J}$-class of $\T(X)$, i.e. 
$J=\{\alpha\in\T(X)\mid |\im(\alpha)|=|X|\}$.  On the other hand, the relative rank of $\T(X)$ modulo the subsemigroup $\OO(X)$ 
was considered by Higgins et al. in \cite{Higgins&Mitchell&Ruskuc:2003}.
They showed that $\rank(\T(X) : \OO(X)) = 1$, when $X$ is an arbitrary countable chain or an arbitrary well-ordered set, 
while $\rank(\T(\R) : \OO(\R))$ is uncountable, by considering the usual order of $\R$. 

\begin{problem}
Determine $\rank(\T(X) : \OP(X))$ and $\rank(\OP(X) : \OO(X))$. 
\end{problem}

Let $J=\{\alpha\in\OO(X)\mid |\im(\alpha)|=|X|\}$. 
Observe that, unlike the analogous set for $\T(X)$, $J$ is not necessarily a 
$\mathcal{J}$-class of $\OO(X)$ (see \cite{Fernandes&al:2014}). 
In \cite{Dimitrova&al:2017}, Dimitrova et al. showed that $\rank (\OO(X) : J) = 0$, for $X\in\{\Z,\Q,\R\}$. 
On the contrary, they proved that  $\rank (\OO(\N) : J) = \aleph_0$. 
In fact, for an infinite countable chain $X$, they gave a necessary and sufficient condition on $X$ for $\rank (\OO(X) : J) = 0$ to hold. 

\begin{problem}
Let $J=\{\alpha\in\OP(X)\mid |\im(\alpha)|=|X|\}$. Determine $\rank(\OP(X) : J)$. 
\end{problem}

\section*{Acknowledgments}
We acknowledge the anonymous referee for his/her valuable suggestions.
Thanks to his/her excellent, careful and in-depth work, this paper has been significantly improved. 
We wish to express him/her our best thanks.


\lastpage 


\begin{thebibliography}{00}

\bibitem{Adams&Gould:1989} 
M.E. Adams and M. Gould, 
Posets whose monoids of order-preserving maps are regular, 
Order 6 (1989), 195--201.

\bibitem{Aizenstat:1962} 
A.Ya. A\u{\i}zen\v{s}tat,
The defining relations of the endomorphism semigroup of 
a finite linearly ordered set,
Sibirsk. Mat. 3 (1962), 161--169 (Russian). 

\bibitem{Aizenstat:1962b} 
A.Ya. A\u{\i}zen\v{s}tat, 
Homomorphisms of semigroups of endomorphisms of ordered sets, 
Uch. Zap., Leningr. Gos. Pedagog. Inst.  238 (1962), 38--48 (Russian). 

\bibitem{Aizenstat:1968} 
A.Ya. A\u{\i}zen\v{s}tat,  
Regular semigroups of endomorphisms of ordered sets, 
Leningrad. Gos. Ped. Inst. U\v cen. Zap. 387 (1968), 3--11 (Russian). 

\bibitem{Araujo&al:2011}
J. Ara\'ujo, V.H. Fernandes, M.M. Jesus, V.Maltcev and J.D. Mitchell, 
Automorphisms of partial endomorphism semigroups, 
Publ. Math. Debrecen 79.1-2 (2011), 23--39. 

\bibitem{Arthur&Ruskuc:2000}
R.E. Arthur and N. Ru\v{s}kuc, 
Presentations for two extensions of the monoid of order-preserving mappings on a finite chain, 
Southeast Asian Bull. Math. 24 (2000), 1--7. 

\bibitem{Catarino:1998}
P.M. Catarino, 
Monoids of orientation-preserving transformations of a finite chain and their presentations,
Proc. of the Conference on Semigroups and Applications, St Andrews, Scotland, 1997 (1998), 39--46.

\bibitem{Catarino&Higgins:1999}
P.M. Catarino and P.M. Higgins,
The monoid of orientation-preserving mappings on a chain,
Semigroup Forum 58 (1999), 190--206. 

\bibitem{Dimitrova&al:2012} 
I. Dimitrova, V.H. Fernandes and J. Koppitz, 
The maximal subsemigroups of semigroups of transformations preserving or reversing the orientation on a finite chain, 
Publ. Math. Debrecen 81.1-2 (2012), 11--29. 

\bibitem{Dimitrova&al:2017} 
I. Dimitrova, V.H. Fernandes and J. Koppitz, 
A note on generators of the endomorphism semigroup of an infinite countable chain, 
J. Algebra Appl. 16 (2017), 1750031 (9 pages). 

\bibitem{Fernandes:1997} 
V.H. Fernandes, 
Semigroups of order-preserving mappings on a finite chain: a new class of divisors,
Semigroup Forum  54 (1997), 230--236.

\bibitem{Fernandes:2002} 
V.H. Fernandes, 
Semigroups of order-preserving mappings on a finite chain: another class of divisors, 
Izvestiya VUZ. Matematika 3 (478) (2002), 51--59 (Russian).

\bibitem{Fernandes&Gomes&Jesus:2009}
V.H. Fernandes, G.M.S. Gomes and M.M. Jesus, 
Congruences on monoids of transformation preserving the orientation on a finite chain,  
J. Algebra 321 (2009), 743--757. 

\bibitem{Fernandes&Gomes&Jesus:2011}
V.H. Fernandes, G.M.S. Gomes and M.M. Jesus, 
The cardinal and the idempotent number of various monoids of transformations on a finite chain, 
Bull. Malays. Math. Sci. Soc. (2) 34 (2011), 79--85. 

\bibitem{Fernandes&al:2014} 
V.H. Fernandes, P. Honyam, T.M. Quinteiro and B. Singha,
On semigroups of endomorphisms of a chain with restricted range,
Semigroup Forum 89 (2014), 77--104.

\bibitem{Fernandes&al:2016} 
V.H. Fernandes, P. Honyam, T.M. Quinteiro and B. Singha,
On semigroups of orientation-preserving transformations with restricted range, 
Comm. Algebra 44 (2016), 253--264. 

\bibitem{Fernandes&al:2010} 
V.H. Fernandes, M.M. Jesus, V. Maltcev and J.D. Mitchell, 
Endomorphisms of the semigroup of order-preserving mappings, 
Semigroup Forum 81 (2010), 277--285. 

\bibitem{Fernandes&Quinteiro:2011}
V.H. Fernandes and T.M. Quinteiro, 
Bilateral semidirect product decompositions of transformation monoids, 
Semigroup Forum 82 (2011), 271--287.

\bibitem{Fernandes&Quinteiro:2012}
V.H. Fernandes and T.M. Quinteiro, 
The cardinal of various monoids of transformations that preserve a uniform partition, 
Bull. Malays. Math. Sci. Soc. (4) 35 (2012), 885--896. 

\bibitem{Fernandes&Quinteiro:2014}
V.H. Fernandes and T.M. Quinteiro, 
On the ranks of certain monoids of transformations that preserve a uniform partition, 
Comm. Algebra 42 (2014), 615--636. 

\bibitem{Fernandes&Volkov:2010} 
V.H. Fernandes and M.V. Volkov, 
On divisors of semigroups of order-preserving mappings of a finite chain, 
Semigroup Forum 81 (2010), 551--554. 

\bibitem{Gomes&Howie:1992}
G.M.S. Gomes and J.M. Howie, 
On the ranks of certain semigroups of order-preserving transformations, 
Semigroup Forum 45 (1992), 272--282. 

\bibitem{Higgins:1995} 
P.M. Higgins, 
Divisors of semigroups of order-preserving mappings on a finite chain, 
Internat. J. Algebra Comput. 5 (1995), 725--742. 

\bibitem{Higgins&Mitchell&Ruskuc:2003}  
P.M. Higgins, J.D. Mitchell and N. Ru\v{s}kuc, 
Generating the full transformation semigroup using order preserving mappings, 
Glasg. Math. J. 45 (2003), 557--566.

\bibitem{Howie:1971} 
J.M. Howie, 
Products of idempotents in certain semigroups of transformations,  
Proc. Edinb. Math. Soc. (2) 17 (1971), 223--236. 

\bibitem{Howie:1995} 
J.M. Howie, 
Fundamentals of Semigroup Theory, 
Oxford, Oxford University Press, 1995.

\bibitem{Howie&Ruskuc&Higgins:1998}  
J.M. Howie, N. Ru\v{s}kuc and P.M. Higgins, 
On relative ranks of full transformation semigroups, 
Comm. Algebra 26 (1998), 733--748.

\bibitem{Kemprasit&Changphas:2000}
Y. Kemprasit and T. Changphas, 
Regular order-preserving transformation semigroups, 
Bull. Austral. Math. Soc. 62 (2000), 511--524. 

\bibitem{Kim&Kozhukhov:2008}
V.I. Kim and I.B. Kozhukhov, 
Regularity conditions for semigroups of isotone transformations of countable chains (Russian), 
Fundam. Prikl. Mat. 12 (2006), no. 8, 97--104; 
translation in J. Math. Sci. (N. Y.) 152 (2008), no. 2, 203--208. 

\bibitem{McAlister:1998} 
D.B. McAlister, 
Semigroups generated by a group and an idempotent, 
Comm. Algebra 26 (1998), 515--547. 

\bibitem{Mora&Kemprasit:2010} 
W. Mora and Y. Kemprasit,
Regular elements of some order-preserving transformation semigroups, 
Int. J. Algebra 4 (2010), 631--641.

\bibitem{Repnitskii&Vernitskii:2000} 
V.B. Repnitski\u{\i} and A. Vernitskii,  
Semigroups of order preserving mappings, 
Comm. Algebra 28 (2000), 3635--3641. 
       
\bibitem{Repnitskii&Volkov:1998} 
V.B. Repnitski\u{\i} and M.V. Volkov, 
The finite basis problem for the pseudovariety $\mathcal{O}$,  
Proc. Roy. Soc. Edinburgh Sect. A 128 (1998), 661--669. 

\bibitem{Ruskuc:1994} 
N. Ru\v{s}kuc, 
On the rank of completely 0-simple semigroups, 
Math. Proc. Cambridge Philos. Soc. 116 (1994), 325--338.

\bibitem{Vernitskii&Volkov:1995} 
A. Vernitskii and M.V. Volkov, 
A proof and generalisation of Higgins' division theorem for semigroups of order-preserving mappings, 
Izv.vuzov. Matematika, No.1 (1995), 38--44 (Russian).

\bibitem{Zhao&Fernandes:2015} 
P. Zhao and V.H. Fernandes, 
The ranks of ideals in various transformation monoids, 
Comm. Algebra 43 (2015), 674--692. 

\end{thebibliography}
\end{document}